\documentclass[preprint,12pt,sort&compress]{elsarticle}

\usepackage{amssymb,amsmath,amsfonts,amsthm,mathtools}
\usepackage{bbm}
\usepackage{bm} 
\usepackage[margin=1in]{geometry}
\usepackage{setspace} 
\usepackage{enumitem}
\usepackage{xcolor,subcaption}
\usepackage[utf8]{inputenc}
\usepackage{amsthm}
    \newtheorem{theorem}{Theorem}[section]
    
    \newtheorem{proposition}[theorem]{Proposition}
    
    \theoremstyle{definition}
    \newtheorem{definition}[theorem]{Definition}
    \newtheorem*{assumptions*}{Assumptions}
\usepackage[mathlines]{lineno} 

\DeclareMathAlphabet\mathbfcal{OMS}{cmsy}{b}{n}

\begin{document}
\begin{frontmatter}

\title{Tipping in a Low-Dimensional Model of a Tropical Cyclone}

\author[1]{Katherine Slyman\corref{cor1}}
\cortext[cor1]{Corresponding Author}
\ead{katherine_slyman@brown.edu}
\author[2]{John A. Gemmer}
\author[3]{Nicholas K. Corak}
\author[4]{Claire Kiers}
\author[4,5]{Christopher K.R.T. Jones}

\address[1]{Division of Applied Mathematics, Brown University, Providence, RI 02906}
\address[2]{Department of Mathematics, Wake Forest University, Winston-Salem, NC 27109}
\address[3]{Department of Physics, Wake Forest University, Winston-Salem, NC 27109}
\address[4]{Department of Mathematics, University of North Carolina at Chapel Hill, Chapel Hill, NC 27599}
\address[5]{Department of Mathematics, George Mason University, Fairfax, VA 22030}


\date{\today}

\begin{abstract}
A presumed impact of global climate change is the increase in frequency and intensity of tropical cyclones. Due to the possible destruction that occurs when tropical cyclones make landfall, understanding their formation should be of mass interest. In 2017, Kerry Emanuel modeled tropical cyclone formation by developing a low-dimensional dynamical system which couples tangential wind speed of the eye-wall with the inner-core moisture. For physically relevant parameters, this dynamical system always contains three fixed points: a stable fixed point at the origin corresponding to a non-storm state, an additional asymptotically stable fixed point corresponding to a stable storm state, and a saddle corresponding to an unstable storm state. The goal of this work is to provide insight into the underlying mechanisms that govern the formation and suppression of tropical cyclones through both analytical arguments and numerical experiments. We present a case study of both rate and noise-induced tipping between the stable states, relating to the destabilization or formation of a tropical cyclone. While the stochastic system exhibits transitions both to and from the non-storm state, noise-induced tipping is more likely to form a storm, whereas rate-induced tipping is more likely to be the way a storm is destabilized, and in fact, rate-induced tipping can never lead to the formation of a storm when acting alone. For rate-induced tipping acting as a destabilizer of the storm, a striking result is that both wind shear and maximal potential velocity have to increase, at a substantial rate, in order to effect tipping away from the active hurricane state. For storm formation through noise-induced tipping, we identify a specific direction along which the non-storm state is most likely to get activated.

\end{abstract}


\begin{keyword}
Dynamical systems \sep rate-induced tipping \sep  Freidlin-Wentzell rate functional \sep noise-induced tipping \sep rare events

\MSC 37H10 \sep 37J45
\end{keyword}

\end{frontmatter}

\section{Introduction}
\label{Intro}
Tipping is the rapid, and often irreversible, change in the state of a system \cite{ashwin2017parameter} and it is well understood that many elements of the climate system are particularly susceptible to tipping in some fashion \cite{Lenton08,lenton2019climate}. There are other reasons to be concerned about tipping with regards to climate change: it has played a role in the collapse of human societies, it exacerbates infectious disease spread and spillover risk, and it affects the severity of extreme weather events \cite{kemp_climate_2022}. While much of the recent mathematical research on tipping has focused on climate applications, see for instance \cite{berglund2002metastability, Eisenman09, Wieczorek10, Eisenman12, rothman2019characteristic, arnscheidt2020routes, lohmann2021risk}, it also has broad applications in  ecology \cite{vanselow2019very, o2020tipping}, ecosystems \cite{drake2010early, scheffer2008pulse}, epidemiology \cite{forgoston2011maximal, schwartz2011converging, Billings2018, hindes2018rare}, and social systems \cite{andreoni_predicting_2021,andreoni_social_2021}. 
Due to the diversity of important applications and the magnitude of the impacts of these phenomena, understanding the mathematics of tipping promises to have significant impact on many existing and recurring problems in our current society. 

We present a study of the determination and classification of tipping events for a low-dimensional model of tropical cyclones. In this system, tipping events can be loosely defined as occurring when a sudden or small changes to a variable or parameter induces a large change to the state of the system in a short amount of time, e.g. the formation or destabilization of a storm. More precisely, in \cite{Ashwin12} it was proposed that tipping events could be predominately classified and studied from a mathematical perspective, according to whether they are induced by a classical bifurcation (B-tipping), a rate dependent parameter (R-tipping), or by noise (N-tipping); see Figure \ref{fig:schematic} for a simplified schematic of these three classifications. These tipping mechanisms do not always act independently; a combination of different mechanisms can also lead to tipping. We specifically explore how parameter shifts and noise affect tipping within tropical cyclones. While in this work we provide a brief overview for these tipping mechanisms, we point the reader to \cite{Thompson11review, lenton2011early, Lenton12, ritchie2016early, ashwin2017parameter, Ashwin12, AshwinThresholds, forgoston2018primer, berglund2013kramers, freidlin2012random} for a more thorough discussion. 

Tropical cyclones, or hurricanes as they are referred to in the Northern Atlantic and Eastern Pacific basins, are complex storms characterized by their rapid rotation and heavy rains and are some of the most costly of natural disasters, both in terms of property damages and in lives lost \cite{mori2021recent}. Hurricane Dorian, one of the strongest tropical cyclones to make landfall in recent years, hit the Bahamas as a category $5$ hurricane, sustaining winds over $185$ miles per hour. It was responsible for an estimated $\$$7 billion in damages, over 400 dead or missing persons, and immeasurable losses to reef and mangroves, which in turn impacted tourism, the fishing industry, and protection from future storms \cite{dahlgren2020preliminary}. In 2005, Hurricane Katrina struck the gulf coast of Louisiana and was one of the costliest storms on record, causing over $\$$125 billion, and over 1,800 lives lost despite only sustaining winds of 127 mph upon landfall \cite{graumann2006hurricane}. Across the globe, tropical cyclones can cause even more damage. It is estimated that the Philippines spend 5$\%$ of its GDP per year on damages from typhoons (tropical cyclones of the  Pacific basin)\cite{bengtsson2001hurricane}. A presumed impact of global climate change is the increase in frequency and intensity of tropical cyclones, e.g., if waters warm north of the equator, it could impact the frequency of tropical cyclone development and, in turn, locations of landfall. As an example, consider Hurricane Lorenzo of 2019 which made landfall in Ireland and was the easternmost Category 5 hurricane on record having impacts across the Atlantic \cite{zelinsky2019tropical}. Because of the destruction that can occur when tropical cyclones make landfall, understanding what mechanisms lead to their formation should be of interest to governments, risk analysts, as well as climate scientists.

\begin{figure}[h]
    \centering
    \includegraphics[scale=.2]{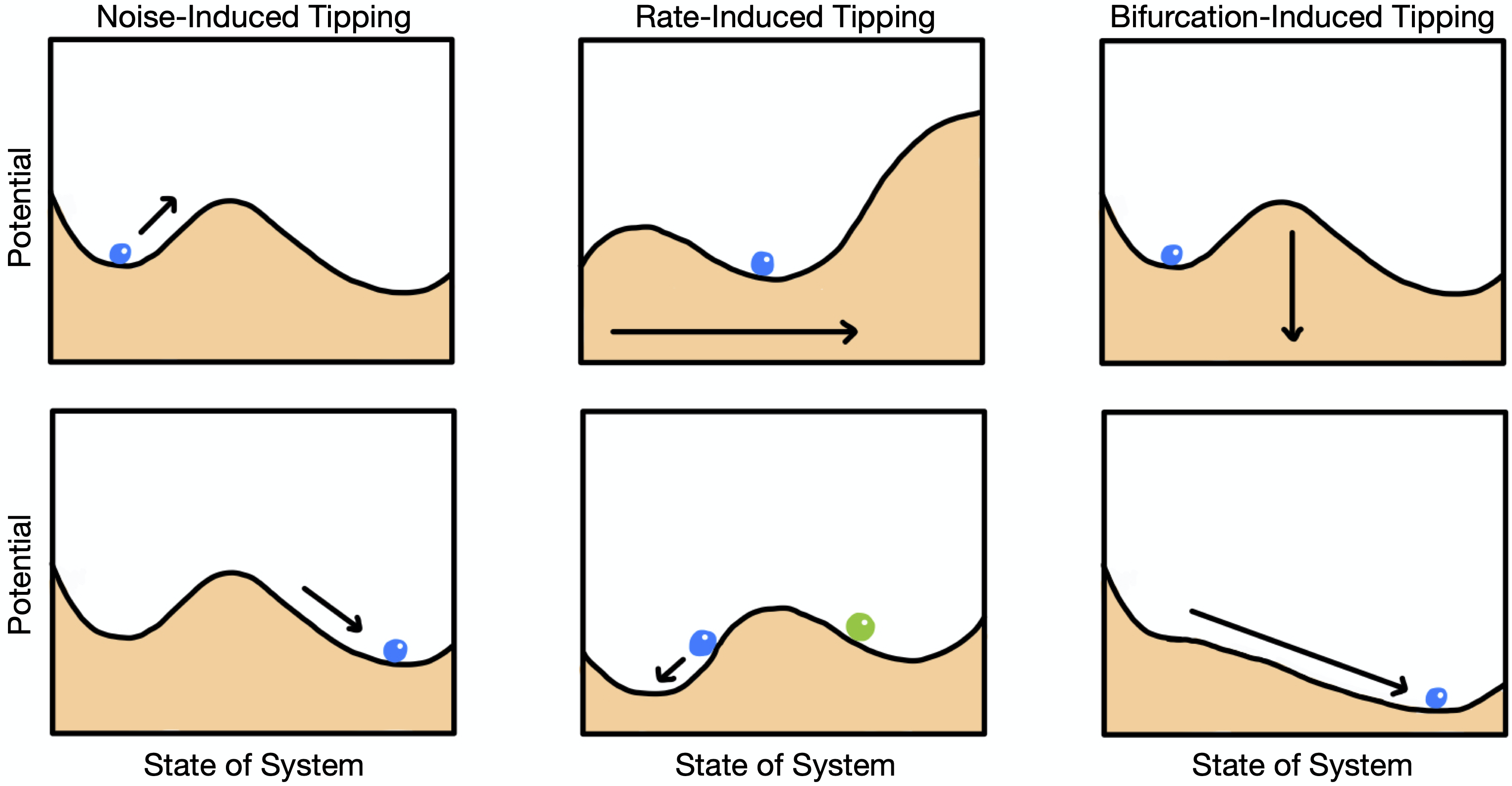}
    \caption{Schematic diagram for noise, rate, and bifurcation-induced tipping in gradient systems when initializing with a particle at a minimum of the potential. For noise-induced tipping, the random fluctuations are needed for the system to overcome the energy barrier to move to another local minimum. In rate-induced tipping, the potential is moved horizontally at a quick enough rate so that particle enters the basin of attraction of another local minimum. In bifurcation-induced tipping, the changing parameter eliminates the energy barrier allowing to particle to move to foll to the other local minimum. This figure is inspired and then recreated from \cite{van_der_bolt_understanding_2021}.}
    \label{fig:schematic}
\end{figure}

\subsection{Description of Model}

Tropical cyclones can be modeled as an axisymetric vortex in hydrostatic equilibrium with a rotational velocity resulting from conservation of angular momentum \cite{emanuel1991theory, emanuel2006hurricanes}. Specifically, tropical cyclones form over warm water, generally between the Tropics of Capricorn and Cancer, in which there is a temperature gradient between the warm ocean and cooler lower atmosphere. Essentially, as warm water evaporates, the resulting warm air mass rises and cools rapidly releasing heat through condensation back into the atmosphere. As the warm air rises, an area of low pressure forms and air begins to move from all directions to fill this void. The air in this region swirls from the Coriolis effect and, due to conservation of angular momentum, eventually forms a rotating air mass around the area of low pressure, i.e., the eye of the storm. 

Formulating a model capturing this effect, \cite{Eman2012} derived an equation for the tangential wind speed $V\geq 0$ resulting from the competition between the dissipation of kinetic energy and the power generated by the storm:
\begin{equation}
\frac{dV}{dt}=\frac{1}{2}\frac{C_d}{h}\left(V_{p0}^2-V^2\right), \label{Eqn:Intro:IdealModel}
\end{equation}
where $C_d/h>0$ has units of inverse length and couples the effect of surface drag and ocean boundary layer depth, i.e. the top depth of ocean that interacts with the bottom layer of the atmosphere \cite{emanuel2006hurricanes}. The maximum potential velocity of the hurricane, $V_{p0}>0$, is a parameter that is calculated by modeling the storm as a Carnot engine and equating the kinetic energy with the theoretical maximum power that could be sustained by the storm \cite{emanuel2006hurricanes}.

The model presented in Equation \eqref{Eqn:Intro:IdealModel} is limited in its applicability, as it does not account for environmental wind shear, dissipative heating, and surface saturation specific humidity. A more realistic model accounting for these effects was developed in the work of Emanuel and Zhang \cite{emanuel2017role}, and Emanuel \cite{emanuel_fast_2017}, and is given by
\begin{equation}
\begin{aligned}
\frac{dV}{dt} &=\frac{1}{2}\frac{C_d}{h}[(1-\gamma) V_p^2m^3-(1-\gamma m^3)V^2], \\
\frac{dm}{dt} &=\frac{1}{2}\frac{C_d}{h}[(1-m)V-2.2Sm],
\end{aligned} \label{Eqn:DimensionAut}
\end{equation}
where $\gamma\in [0,1]$ is a dimensionless parameter accounting for dissipative heating and pressure dependence of the surface saturation humidity, $V_p^2=(1-\gamma)^{-1}V_{p0}^2$, and $S$ is the wind shear measured in units of velocity. More specifically, $\gamma=(T_A-T_0)/T_0+\kappa$, where $T_A,T_0$ are the temperatures of the lower atmosphere and upper ocean respectively, $\kappa$ is a constant and thus $\gamma^{-1}$ is a proxy for the temperature of the ocean. Here, the dependent variable $m$ can be thought of as relative humidity, thus dimensionless and satisfies $m\in [0,1]$. In this model, $m$ serves as a ``fuel" for the tropical storm, and indeed if $m=1$, i.e. the core is fully saturated, we recover Equation \eqref{Eqn:Intro:IdealModel}. In contrast, $S$ is coupled in such a way as to reflect wind shear's role in pulling moisture out of the storm and thus possibly dissipating it. This form of the model, in particular the factor of $2.2$ and the cubic nonlinearities, are empirical and are the result of numerical experimentation \cite{emanuel2017role}. Note, in \cite{emanuel2017role, emanuel_fast_2017} this model is presented with an ocean feedback term accounting for the fact that high velocity storms begin to pull up ocean water which cools the hurricane. This additional term does not significantly impact the qualitative behavior of this model and thus to simplify the analysis we have chosen to neglect it. Note, for physically relevant parameters and nonzero wind shear, this dynamical system always contains a stable fixed point at the origin $(V=0,m=0)$ corresponding to a non-storm state and for sufficiently low wind shear, an additional asymptotically stable fixed point corresponding to a stable storm state, and a saddle corresponding to an unstable storm state.  

\subsection{Summary of Key Results and Organization of Paper}

In Section \ref{DetModel}, we examine the dimensionless version of Equation \eqref{Eqn:DimensionAut}, lay out the parameter regimes of interest, perform a standard bifurcation analysis, and study the qualitative behavior of the model. 

In Section \ref{Rate}, we present the basic theoretical framework for rate-induced tipping, study the possibility of rate-induced transitions between the non-storm and stable storm states and conclude this section with a numerical example that illustrates rate-induced tipping. Creating a storm by tipping from the non-storm state to the stable storm state, is never possible through rate-induced tipping alone due to the fact the non-storm state is stationary with respect to changes in parameters.

We prove that the deterministic system undergoes rate-induced tipping away from the stable storm state to the non-storm state. This will occur when the max potential velocity and the dimensionless wind shear both increase at a sufficiently high rate. A surprising and very interesting aspect of this result is that both maximal velocity and wind shear need to increase to effect the tipping. It is not surprising that wind shear needs to increase as it is well-known that a lack of wind shear is necessary to support a hurricane. The maximum potential velocity of the storm is a proxy for the energy available to the storm and it is counter-intuitive that it should have to increase to force a storm to end, rather than the other way around. While this may be revealing a hidden flaw in such a low-dimensional model, it may also be a genuine effect in that increasing the maximum potential velocity is implicitly requiring the storm to be stronger to survive and if this requirement sets in too quickly then the storm may be unable to adjust. Therefore, in this model at least, a hurricane can destabilize due to rapidly changing parameters, but can never form.
 
In Section \ref{Noise}, we consider the addition of noise to the system and investigate transitions between the non-storm and stable storm states. We show the stochastic system exhibits tipping to and from the non-storm state, implying a hurricane can form or destabilize with the addition of random fluctuations acting on the system. The primary mathematical tool we use is the Freidlin–Wentzell theory of large deviations to determine the most probable transition path between states. In this framework, the most probable transition path can be computed as the minimum of an action functional which satisfies a Hamiltonian system of differential equations. By exploiting the Hamiltonian structure of these equations we are able to compute an asymptotic formula for the most probable transition path from the non-storm state to the stable storm state which further allowed us to estimate the expected tipping time from the non-storm state. To validate this approximation we compared our approximation with a direct gradient flow of the action functional. Monte Carlo simulations reveal the accuracy of the most probable path, and also demonstrate the system's susceptibility to tipping. 

However, while it is truly a rare event for noise to tip the system from the stable storm state to the non-storm state, we show that the stochastic system is highly susceptible to tip from the non-storm state to the stable storm state. Both analytical arguments and numerical experiments show that noise-induced tipping is needed (and likely) to form a storm whereas rate-induced tipping is a favored mechanism for destabilizing a storm. 

Lastly, in Section \ref{Summary}, we discuss the implications and key significance of this work and 
discuss the interplay between the rate and noise-induced tipping mechanisms via numerical
simulations. When considering the stochastic system, and also allowing both the max potential velocity and the dimensionless wind shear to be time-dependent parameters, the system again exhibits tipping to and from the non-storm state. When the system tips from the non-storm state to the stable storm state, the two tipping mechanisms act independently: noise-induced tipping occurs before the ramp begins and then the system end-point tracks the stable storm state, but when the system tips from the stable storm state to the non-storm state, there is an interplay of the two tipping mechanisms.

\section{Analysis of Autonomous System and Bifurcation-Induced Tipping} \label{DetModel}

To give context to the tipping results, we first perform a standard bifurcation analysis of Equation \eqref{Eqn:Intro:IdealModel} as well as study the qualitative behavior of the model. To do so, it convenient to introduce the dimensionless variables $v=V/V_p$ and $\tau=C_d/(2h)V_p t$ which yields the dimensionless system given by
\begin{equation}
\begin{aligned}
    \frac{dv}{d \tau} &= f(v,m):=(1-\gamma)m^3-(1-\gamma m^3)v^2, \\
   \frac{dm}{d \tau} &= g(v,m):=(1-m)v-c m,
\end{aligned} \label{Eqn:DimensionlessAut}
\end{equation}
where $c=2.2S/V_p$ and is the dimensionless wind shear. Note, from the discussion of this model in the introduction, we expect that $0\leq V\leq V_p$ and $0\leq m \leq 1$ and thus the relevant phase space on which the forward flow of Equation \eqref{Eqn:DimensionlessAut} should be defined on is $\Pi=[0,1]\times [0,1]$. Indeed, since $\gamma\in (0,1)$ it follows that $f(0,m)\geq 0$, $f(1,m)\leq 0$, $g(v,0)\geq 0$, $g(v,1)\leq 0$ and thus $\Pi$ is forward invariant. 

The fixed points of Equation \eqref{Eqn:DimensionlessAut} satisfy the system of equations $m=v/(v+c)$ and $q(v)=v^2p(v)=0$, where $p$ is the cubic polynomial
\begin{equation}
    p(v)=(1-\gamma)v+\gamma v^3-(v+c)^3. \label{poly}
\end{equation}
Therefore, regardless of parameter values, the origin $\mathcal{O}=(0,0)$ (non-storm state) is a fixed point of Equation \eqref{Eqn:DimensionlessAut}. However, $v=0$ is a repeated root of the quintic polynomial $q$ and thus $\mathcal{O}$ is not a hyperbolic fixed point. Nevertheless, it can be shown through a center manifold reduction that the origin is in fact a stable fixed point; see Appendix B. 

Additional fixed points for Equation \eqref{Eqn:DimensionlessAut} can exist depending on the parameter values. Specifically, since $\lim_{v \rightarrow \infty} p(v)=-\infty$ and $p(0)=-c^3$, it follows that, including repeated roots, $p$ can either have zero or two positive roots. Consequently, this system can exhibit a saddle-node bifurcation in which a stable node $\mathcal{S}$ (stable storm state) and a saddle $\mathcal{U}$ (unstable storm state) emerge or disappear when varying a parameter; see Figure \ref{Fig:bifurcation}(a) for an example bifurcation diagram in the parameter $c$. The saddle-node bifurcation occurs when the local maximum $v^*$ of $p(v)$ intersects the $v$-axis, i.e., $p(v^*)=0$. Calculating $v^*$ we find 
\begin{equation}
v^*=\frac{-c+\sqrt{\frac{1}{3}(1-\gamma)^2+c^2\gamma}}{1-\gamma},
\end{equation}
and thus as functions of $c$ and $\gamma$ the curve along which the bifurcations occur are given as the $0$-level curve of $p(v^*)$. In Figure \ref{Fig:bifurcation}(b) we plot a ``phase diagram'' in which the the $0$-level curve of $p(v^*)$ partitions the $(c,\gamma)$ plane into regions in which a stable storm state does (labeled Existence of Storm State) or does not exist (labeled Existence of Non-Storm State). 

\begin{figure}[t!]
\centering
\begin{subfigure}[b]{0.4\textwidth}
    \centering
    \includegraphics[width=.95\textwidth]{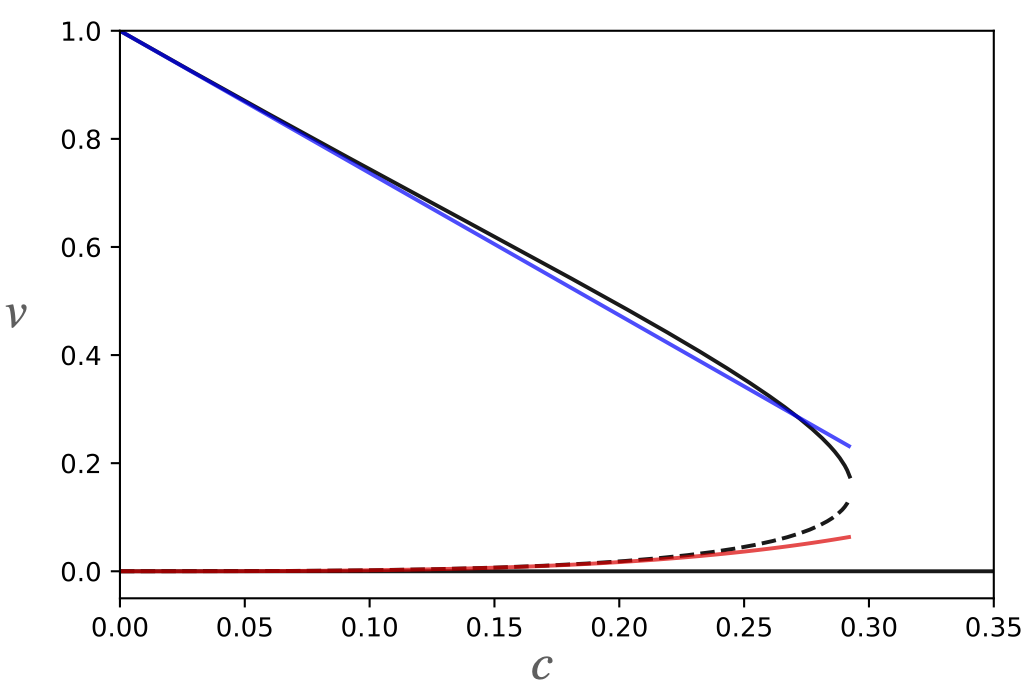}
    \caption{}
\end{subfigure}
\begin{subfigure}[b]{0.4\textwidth}
    \centering
    \includegraphics[width=.95\textwidth]{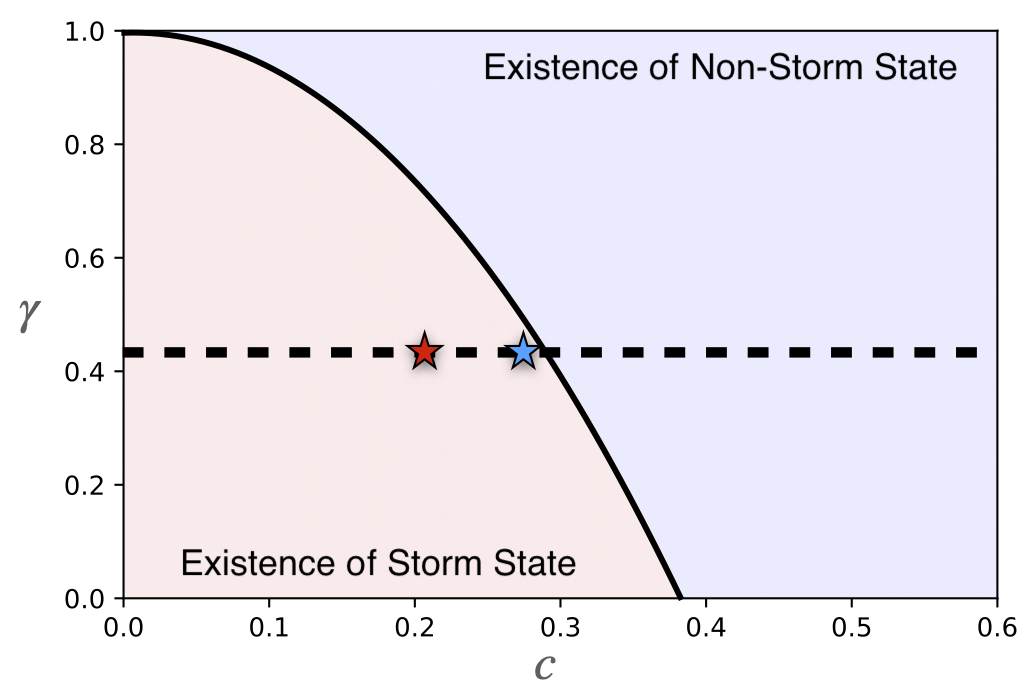}
    \caption{}
\end{subfigure}\\
\begin{subfigure}[b]{.4\textwidth}
    \centering
\includegraphics[width=.95\textwidth]{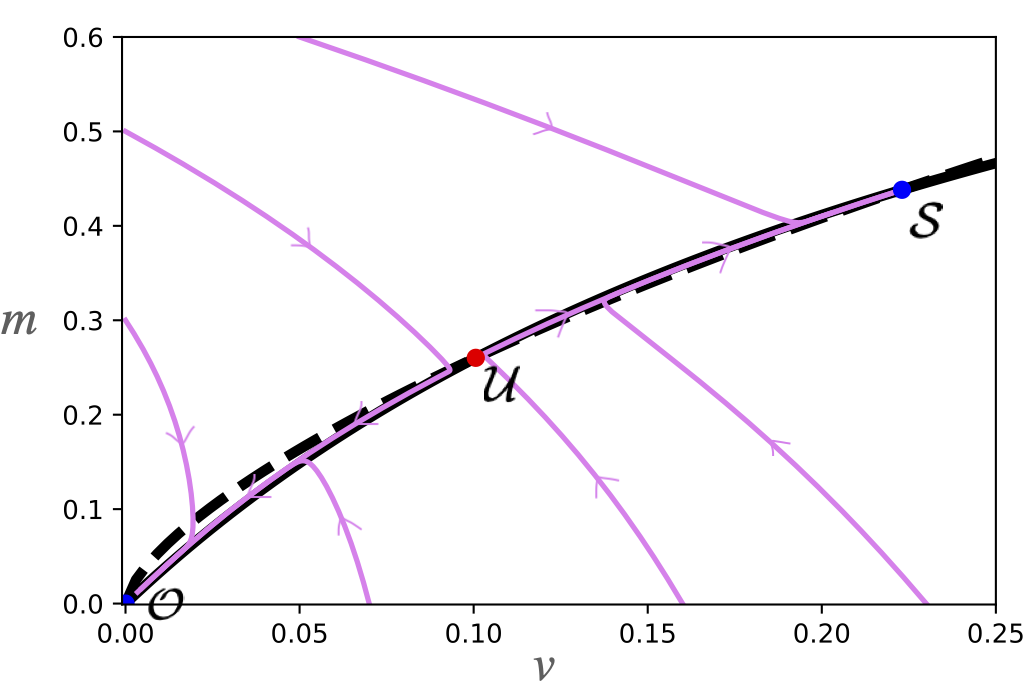}
\caption{}
\end{subfigure}
\begin{subfigure}[b]{.4\textwidth}
    \centering
\includegraphics[width=.95\textwidth]{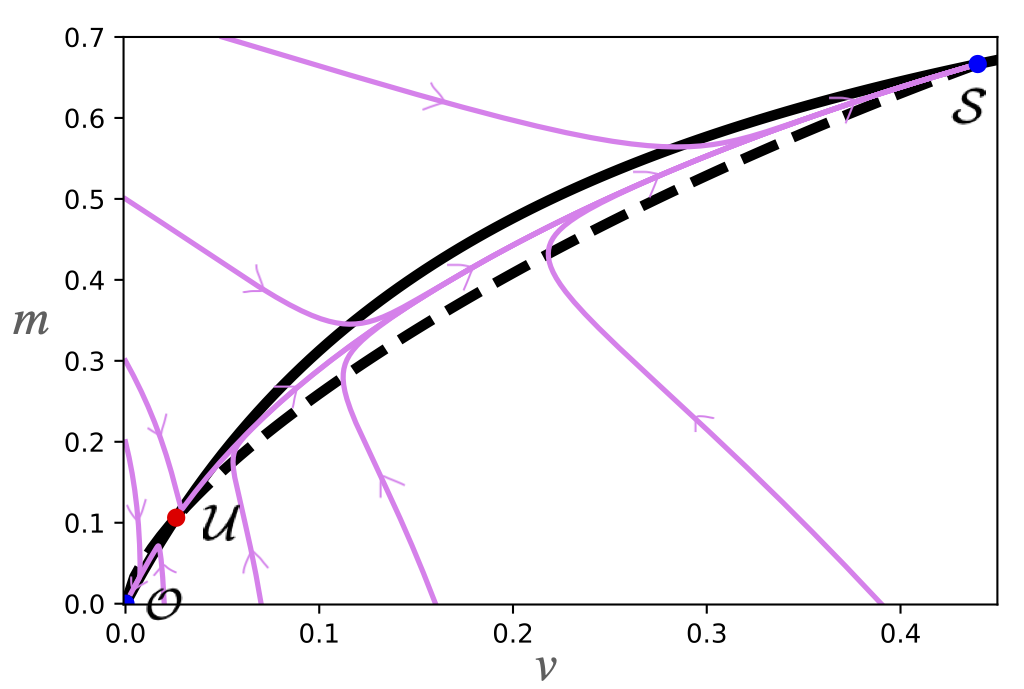}
\caption{}
\end{subfigure}
\begin{subfigure}[b]{.4\textwidth}
    \centering
\includegraphics[width=.95\textwidth]{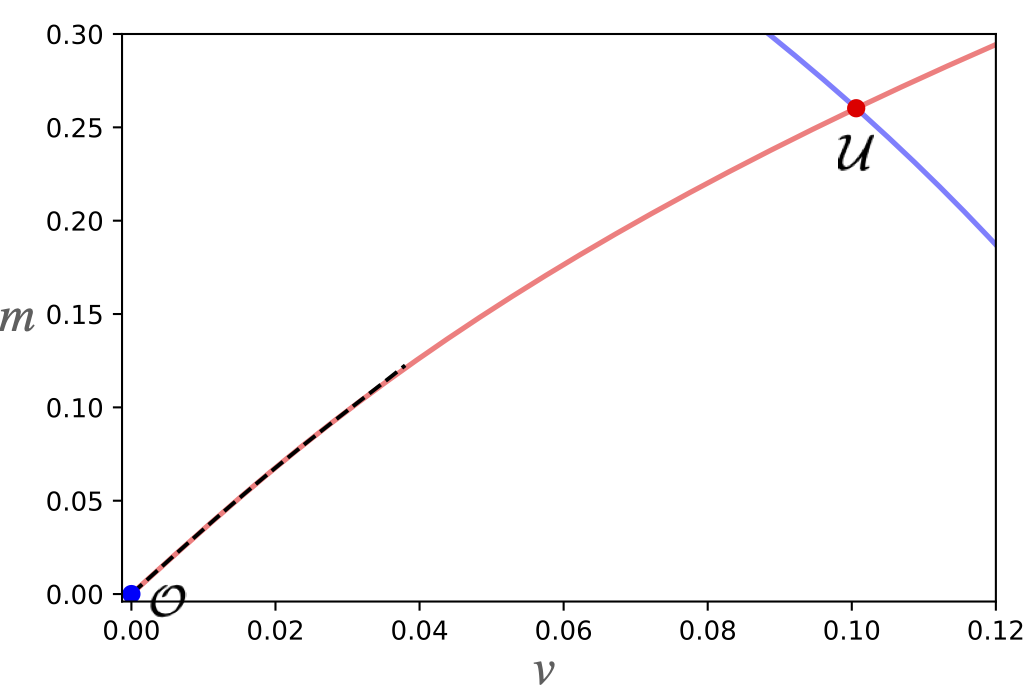}
\caption{}
\end{subfigure}
\begin{subfigure}[b]{.4\textwidth}
    \centering
\includegraphics[width=.96\textwidth]{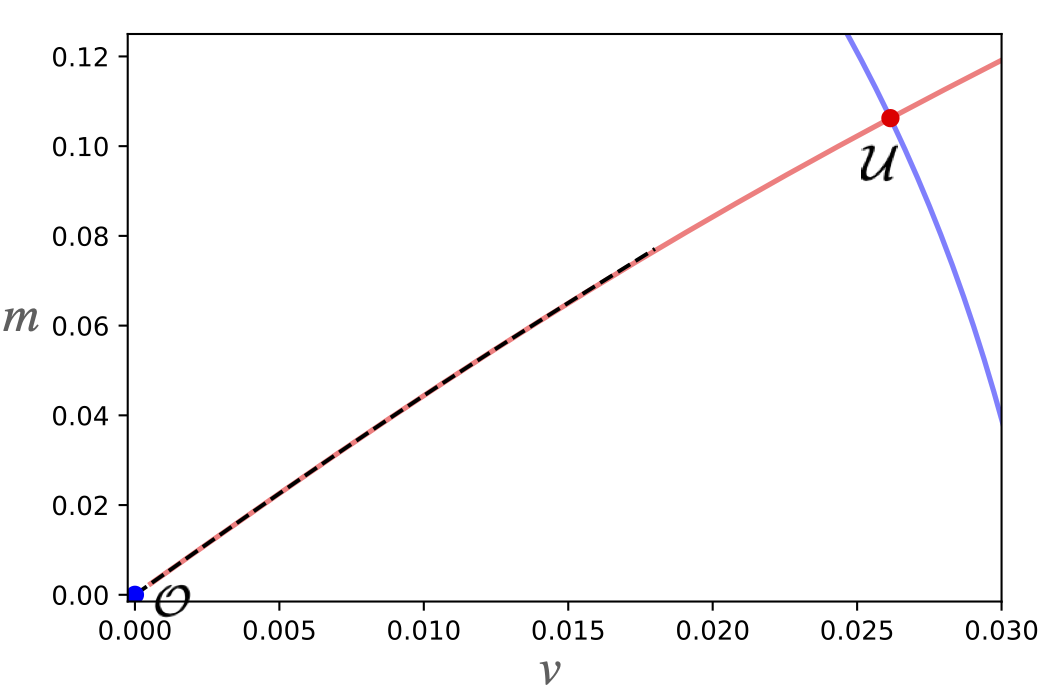}
\caption{}
\end{subfigure}
\caption{(a) Bifurcation diagram for Equation \eqref{Eqn:DimensionlessAut}  as a function of dimensionless wind shear $c$ with $\gamma=0.43$. There are three possible fixed points: the asymptotically stable origin $\mathcal{O}$ (non-storm state), another asymptotically stable fixed point $\mathcal{S}$ (stable storm state), and a saddle $\mathcal{U}$ (unstable storm state).  (b) Phase diagram for Equation \eqref{Eqn:DimensionlessAut} in the $(\gamma,c)$ plane in which the regions corresponding to the existence or non-existence of $\mathcal{S}$ and $\mathcal{U}$ are labeled as Existence of Storm State and Non-Storm State respectively. (c-d) Phase portraits for Equation \eqref{Eqn:DimensionlessAut} with $(\gamma,c)=(0.43,0.286)$ and $(\gamma,c)=(0.43,0.22)$ respectively. The blue circles correspond to the stable fixed points $\mathcal{S}$ and $\mathcal{O}$ and the red square corresponds to the saddle node $\mathcal{U}$. The dashed and solid black lines correspond to the nullclines $\dot{v}=0$ and $\dot{m}=0$ respectively. (e-f) Magnifications of the phase portraits in Figures \ref{Fig:bifurcation}(c-d) near the origin. The red (blue) lines correspond to the unstable (stable) manifold of $\mathcal{U}$ and the dashed line is the local approximation of the center manifold near the origin.}
\label{Fig:bifurcation}
\end{figure}

Figures \ref{Fig:bifurcation}(a) and \ref{Fig:bifurcation}(c) indicate that for $c\ll 1$, $\mathcal{U}$ remains close to the origin while the $v$-coordinate of $\mathcal{S}$ varies linearly in $c$. This result can be verified by an asymptotic expansion of the zeros of $p(v)$. Specifically, we assume an asymptotic expansion of a fixed point $(v^*,m^*)$ in the form
\begin{equation}
\begin{aligned}
v^*&=v_0+c^{\alpha_1}v_1+c^{\alpha_2}v_2+\ldots,\\
m^*&=\frac{v^*}{v^*+c},
\end{aligned}
\end{equation}
where $0<\alpha_1<\alpha_2<\ldots$. At lowest order in $c$ we obtain the cubic equation $v_0(1-\gamma)(1-v_0)^2=0$ and thus the non-negative roots are $v_0=0$ and $v_0=1$ which correspond to the lowest order approximations of $\mathcal{U}$ and $\mathcal{S}$ respectively. 
\begin{enumerate}
\item If we continue the expansion assuming $v_0=0$, we find that to balance terms we need $\alpha_1=3, \alpha_2=5$, $v_1=(1-\gamma)^{-1}$, $v_2=3(1-\gamma)^{-2}$ and therefore
\begin{equation}
m^*=\frac{c^3(1-\gamma)^{-1}+3c^5(1-\gamma)^{-2}+\ldots}{c+c^3(1-\gamma)^{-1}+3c^5(1-\gamma)^{-2}+\ldots}=\frac{c^2}{1-\gamma}+\frac{2c^4}{(1-\gamma)^2}+\ldots.
\end{equation}
\item Assuming $v_0=1$ we obtain a regular perturbation, i.e., $\alpha_1=1, \alpha_2=2, \ldots$, implying $v_1=-\frac{3}{2}(1-\gamma)^{-1}$, $v_2=-\frac{3}{8}(1-4\gamma)(1-\gamma)^{-2}$ and thus
\begin{equation}
m^*=\frac{1-\frac{3}{2}(1-\gamma)^{-1}c+\ldots}{1-\left(\frac{3}{2}(1-\gamma)^{-1}c-c\right)}=1-c+\ldots.
\end{equation}
\end{enumerate}
Therefore, the first two nonzero terms in the asymptotic expansions for $\mathcal{U}$ and $\mathcal{S}$ in $c$ are given by
\begin{equation}
\begin{aligned}
\mathcal{U}&=\left(\frac{c^3}{1-\gamma}+\frac{3c^5}{(1-\gamma)^2},\frac{c^2}{1-\gamma}+\frac{2c^4}{(1-\gamma)^2}\right)+\left(o(c^5),o(c^4)\right),\\
\mathcal{S}&=\left(1-\frac{3}{2}(1-\gamma)c, 1-c\right)+\left(o(c^2),o(c^2)\right).
\end{aligned} \label{Eqn:DeterministicAsyExpansionFP}
\end{equation}

The asymptotic expansion in Equation \eqref{Eqn:DeterministicAsyExpansionFP} indicates that for weak wind shear the separation in phase space between the non-storm state $\mathcal{O}$ and the unstable storm state $\mathcal{U}$ is small in comparison with the separation between $\mathcal{U}$ and the stable storm state $\mathcal{S}$. This separation is illustrated by the generic phase portraits of Equation \eqref{Eqn:DimensionlessAut} presented in Figures \ref{Fig:bifurcation}(c-d) for parameter values in which $\mathcal{S}$ exists. To indicate the basins of attraction for $\mathcal{O}$ and $\mathcal{S}$, denoted $\mathbb{B}(\mathcal{O})$ and $\mathbb{B}(\mathcal{S})$ respectively, in Figures \ref{Fig:bifurcation}(e-f) we plot the unstable and stable manifolds of $\mathcal{U}$ near the origin. As it will be important in Section \ref{Noise} when we discuss noise-induced tipping, note that the stable manifold of $\mathcal{U}$ forms a separatrix between $\mathbb{B}(\mathcal{O})$ and $\mathbb{B}(\mathcal{S})$.

A final feature of this dynamical system is the existence of a center manifold near $\mathcal{O}$ which is approximated to cubic order by
\begin{equation}
    m(v)=\frac{1}{c}v-\frac{1-\gamma}{c^5}v^3;
\end{equation}
see Appendix A for the derivation. This approximation is overlayed as a dashed curve in Figures \ref{Fig:bifurcation}(e-f). This manifold acts as a slow manifold near $\mathcal{O}$ in the sense that, near $\mathcal{O}$, the component of the vector field transverse to the manifold is much larger in magnitude than the component tangent to the manifold. Consequently, the center manifold provides a natural pathway for noise-induced transitions from $\mathcal{O}$ to $\mathcal{S}$ to occur; a conjecture we will verify in Section \ref{Noise}.

\section{Rate-Induced Tipping in the Tropical Cyclone Model} \label{Rate}
Of central interest is the possibility of transitions between the two stable states as the max potential velocity, $V_p$, and the dimensionless wind shear, $c$, vary in time. Rate-induced tipping is where a sufficiently quick change to a parameter of a system may cause the system to move away from one attractor to another, without undergoing a bifurcation \cite{ashwin2017parameter}. Essentially, the system is unable to track a continuously changing attractor if the parameter changes fast enough. 

\subsection{A Quick Introduction to Rate-Induced Tipping} \label{sec:quick intro rtip}

Following the work done by \cite{Ashwin12,ashwin2017parameter,AshwinThresholds}, we lay out the framework needed to describe rate-induced tipping more formally and also introduce notation we use throughout this work.

Consider the autonomous differential equation
\begin{equation}
    \dot{x}=f(x,\lambda),
\label{EQ: rate_tip_setup}
\end{equation}
where $x\in \mathbb{R}^n, \lambda \in \mathbb{R}^m, f\in C^2(\mathbb{R}^{m+n},\mathbb{R}^n), t \in \mathbb{R}$, and $\dot{x}$ is the derivative of $x$ with respect to time, $\frac{dx}{dt}$. Now, instead of a fixed $\lambda$, suppose that $\lambda$ changes in time. We replace  $\lambda$ with an external input $\Lambda_r(t)=\Lambda(rt) \in C^2(\mathbb{R},\mathbb{R}^m)$, $r \in \mathbb{R}>0$, and specifically assume that $\Lambda_r$ is bi-asymptotically constant. This implies that $\Lambda_r$ is a parameter shift that satisfies
\begin{center} $\displaystyle \lim_{t \rightarrow  -\infty} \Lambda_r(t)=\lambda_- \in \mathbb{R}^m$ and $\displaystyle \lim_{t \rightarrow  \infty} \Lambda_r(t)=\lambda_+ \in \mathbb{R}^m$,
\end{center}
where $\lambda_-$ the past limit state and $\lambda_+$ is the future limit state. In addition to assuming that $\Lambda_r$ is bi-asymptotically constant, assume that $\Lambda_r$ is monotonically increasing, and that 
\begin{center}$\lambda_-<\Lambda_r<\lambda_+$. \end{center}

\noindent These assumptions on $\Lambda_r$ allow a gradual transition between $\lambda_-$ and $\lambda_+$ in time, where the size of $r$, the rate parameter, determines how quickly $\Lambda_r$ transitions between $\lambda_-$ and $\lambda_+$. While there are different types of functions that fit this criteria, we use a transformation on a hyperbolic tangent function as $\Lambda_r$. Note that the external input $\Lambda_r$ can also be called a ramp function or a ramp parameter.

\begin{definition}
Suppose $\Lambda_r(t)$ is a bi-asymptotically constant external input, with future limit state $\lambda_+$ and past limit state $\lambda_-$. Suppose that for all $t \in \mathbb{R}$, $X(t)$ is a fixed point of Equation \eqref{EQ: rate_tip_setup} when $\lambda=\Lambda_r(t)$, $(t,X(t))$ is a connected curve, and when $\lambda=\lambda_{\pm}, X(t)=X^{\pm}$. Then $(t,X(t))$ is a \textit{stable (unstable) path} if $X(t)$ is an attracting (repelling) fixed point for all $t$. \textit{These paths can be referred to as a path of stable (unstable) fixed points in the frozen time system.}
\end{definition}

Replacing $\lambda$ with $\Lambda_r(t)$ in Equation \eqref{EQ: rate_tip_setup} leads to
\begin{equation}
\dot{x}=f(x,\Lambda_r(t)),
\label{EQ: rate_tip_setup_2}
\end{equation}
where $x\in \mathbb{R}^n, \Lambda_r(t) \in \mathbb{R}^m, f\in C^2(\mathbb{R}^{m+n},\mathbb{R}^n), t \in \mathbb{R}$, and $r \in \mathbb{R}$. 
Solution behaviors of Equation \eqref{EQ: rate_tip_setup_2} change for different values of $r$. Let $x^{r^*}(x_0)$ denote a solution of Equation \eqref{EQ: rate_tip_setup_2} initialized at $x=x_0$ for a given value of $r$. We define rate-induced tipping using the definition from \cite{AshwinThresholds} below.

\begin{definition}
Consider the nonautonomous system given by Equation \eqref{EQ: rate_tip_setup_2} with a bi-asymptotically constant external input $\Lambda_r(t)$, with future limit state $\lambda_+$ and past limit state $\lambda_-$. Suppose that when $\Lambda_r=\lambda_-$, the system has a hyperbolic sink $e^-$, and when $\Lambda_r=\lambda_+$, the system has a compact invariant set $\eta^+$ that is not an attractor. Equation \eqref{EQ: rate_tip_setup_2} undergoes \textit{rate-induced tipping} from $e^-$ if there are rates $0<r_2<r_1$ such that $\lim_{t \rightarrow \infty} x^{r_1} (e^-) \rightarrow \eta^+$ and $\lim_{t \rightarrow \infty} x^{r_2} (e^-) \not\rightarrow \eta^+$. The first value of $r$ such that $\lim_{t \rightarrow \infty} x^{r} (e^-) \rightarrow \eta^+$ is called a critical rate and is denoted by $r_c$.
\end{definition}

\begin{definition}
Suppose that when $\Lambda_r(t)=\lambda$ the system given by Equation \eqref{EQ: rate_tip_setup_2} has a hyperbolic sink $e^-$. We define the basin of attraction for $e^-$ as $\mathbb{B}(x,\lambda)=\{x \hspace{1mm} | \hspace{1mm} \lim_{t \rightarrow \infty} x(t) = e^-\}$ .
\end{definition}

\noindent Essentially, if $r<r_c$, solutions will end-point track the path of fixed points in the frozen time system they were initialized on. However, when $r=r_c$, we have tipped to the basin boundary of $e^-$, and when $r>r_c$, we either tip to infinity or a different attractor, as it left the basin of attraction for $e^-$. It is worth mentioning that not all choices of $\Lambda_r$ result in rate-induced tipping. There is theory to let us know if a system will or will not tip with a chosen $\Lambda_r$, and the conditions change based on the dimension of the system. Nevertheless, the conditions for a system to undergo rate-induced tipping are the same for arbitrary dimension. In particular, rate-induced tipping can occur in the system given by Equation \eqref{EQ: rate_tip_setup_2} if the starting base state satisfies a sufficient condition called \textit{forward threshold unstable} when the parameter shift is applied to it \cite{AshwinThresholds}. Condensing the results of the work by \cite{ashwin2017parameter,Kiers.2019}, the following theorem is relevant for the analysis in this section.

\begin{theorem}
Suppose $\Lambda_r$ gives rise to a stable path $(t,X(t))$ in Equation \eqref{EQ: rate_tip_setup_2} with $x \in \mathbb{R}^n$, for $n \geq 1$, and $X^{\pm}=\lim_{t \rightarrow \infty} X(t)$. When $\Lambda_r=\lambda_-$, assume the system has an attracting equilibrium $X^-$. If there is a $Y^+ \neq X^+$ such that $Y^+$ is an attracting equilibrium of Equation \eqref{EQ: rate_tip_setup} for $\lambda=\lambda^+$, and $X^- \in \mathbb{B}(Y^+,\lambda^+)$, then there is rate-induced tipping away from $X^-$ to $Y^+$ for this $\Lambda_r$, for sufficiently large $r > 0$.
\label{THM: rate induced conditions}
\end{theorem}

\noindent Forward threshold stability is not a necessary condition to prevent rate-induced tipping in systems of dimension higher than one. However in \cite{Kiers.2019}, it was proposed that the condition of inflowing stability guarantees that rate-induced tipping cannot happen away from a stable path. It is the approach we take to study rate-induced tipping in the tropical cyclone model.

\begin{definition}
\label{Def:FIS}
Suppose $\Lambda_r$ gives rise to a stable path $(t,X(t))$ and $X^{\pm}=\lim_{t \rightarrow \pm \infty} X(t)$. We say the stable path $(t, X(t))$ is forward inflowing stable if for each $t \in \mathbb{R}$ there is a compact sets $K(t)$ such that

1. \text{For all} $t \in \mathbb{R}$, $X(t) \in$ Int $K(t)$; 

2. If $t_1<t_2, \text{then } K(t_1) \subset K(t_2)$;

3. If $x \in \partial K(t), \text{then there exists } t_0>0 \text{ such that for all } t \in (0,t_0), x(\Lambda_r(t)) \in \text{Int } K(t);$

4. $\displaystyle X_{\pm} \in \text{Int }(K_{\pm}) \text{ where } \displaystyle K_- = \bigcap_{t \in \mathbb{R}} K(t) \text{ and }K_+ = \overline{\bigcup_{s \in \mathbb{R}} K(t)}; $

5. $K_+ \subset B(X^+,\lambda^+)$ is compact.
\end{definition}

\begin{theorem}
Suppose $\Lambda_r$ gives rise to a stable path $(t,X(t))$ and $X^{\pm}=\lim_{t \rightarrow \infty} X(t)$. If the stable path $(X(t),\Lambda_r(t))$ is forward inflowing stable, then there is no R-tipping away from $X^-$ for this $\Lambda_r$.
\end{theorem}

Finally, we comment that to numerically study Equation \eqref{EQ: rate_tip_setup_2}, we convert the system back into an autonomous equation. The standard approach to investigate rate-induced tipping is to augment the system by introducing a new variable, $s=t$. We make the corresponding substitutions and differentiate, resulting in the two-dimensional autonomous system given by
\begin{equation}
\begin{aligned}
\dot{x}&=f(x,\Lambda_r(s)), \\
\dot{s}&=1.
\label{EQ: numerically_rtip}
\end{aligned}
\end{equation}
While this system technically has no fixed points in $x$-$s$ space, fixed points of Equation \eqref{EQ: rate_tip_setup_2}, correspond to the level sets $\dot{x}=0$ of Equation \eqref{EQ: numerically_rtip}, i.e. $f(x,\Lambda_r(s))=0$.

\subsection{Necessary Conditions for Rate-Induced Tipping in the Tropical Cyclone Model}
\label{RITTCM}
Using the above framework, we consider the possibility of rate-induced tipping to aid in the destabilization or the formation of a storm by allowing $c$ and $V_p$ to vary with time, as physically, both wind shear and max potential velocity are components of the model that have the ability to change quickly.

We redefine both  $c$ and $V_p$ as functions of some parameter shift $\Lambda_r: \mathbb{R} \rightarrow \mathbb{R}$ that varies at a rate $r>0$, and satisfies the conditions described in Section \ref{sec:quick intro rtip}. Allowing $V_p$ and $c$ to be dependent on $\Lambda_r(s)$ implies they will ramp between $V_p^-$ to $V_p^+$ and $c^-$ to $c^+$ in time, respectively. Therefore, we adjust the nondimensionalization using $\tau=C_d/(2h)V_p^- t$ and $v=V/V_p^-$, instead of $V_p$, as this parameter is now time dependent. We choose the minimum max storm potential, $V_p^-$, as the fixed value of $V_p(\tau)$. This change results in Equation \eqref{Eqn:DimensionAut} becoming

\begin{equation}
\begin{aligned}
    \frac{dv}{d \tau} &= \frac{(1-\gamma) V_p(\tau)^2}{{V_p^-}^2} m^3-(1-\gamma m^3)v^2, \\
    \frac{dm}{d \tau} &= (1-m)v-c(\tau)m.
\end{aligned}\label{Eqn:DimensionlessNonAut}
\end{equation}

\begin{proposition}
There is no rate-induced tipping away from $\mathcal{O}$ regardless of the $\Lambda_r$ chosen.
\label{Prop: No r-tip}
\end{proposition}

\begin{proof}
For all values of $c$ and $V_p$, $\mathcal{O}$ is a stable fixed point and $\dot{v}=\dot{m}=0$. Therefore, there can be no rate-induced tipping away from $\mathcal{O}$ for any ramp $\Lambda_r$.
\end{proof}

We will generally assume the values of $c$ and $V_p$ are such that there are three fixed points of Equation \eqref{Eqn:DimensionlessNonAut} for all values of $\Lambda_r(\tau)$, as this is the most interesting case to study due to Proposition \ref{Prop: No r-tip}. However, we describe the other cases and their outcomes below.

\

\noindent \textit{Case 1.} Assume that the values of $c$ and $V_p$ are chosen so there is only one fixed point of Equation \eqref{Eqn:DimensionlessNonAut} when $\Lambda_r(\tau)=\lambda^-$, and one or three fixed points of Equation \eqref{Eqn:DimensionlessNonAut} when $\Lambda_r(\tau)=\lambda^+$. If there is only one fixed point when $\Lambda_r(\tau)=\lambda^-$, it has to be the origin, $\mathcal{O}$, as shown in the deterministic analysis conducted in Section \ref{DetModel}. By Proposition \ref{Prop: No r-tip}, there can be no rate-induced tipping away from $\mathcal{O}$.

\

\noindent \textit{Case 2.} Assume that the values of $c$ and $V_p$ are chosen so there are three fixed points of Equation \eqref{Eqn:DimensionlessNonAut} when $\Lambda_r(\tau)=\lambda^-$, $\mathcal{O},\mathcal{U}^-,\mathcal{S}^-$, and one fixed point of Equation \eqref{Eqn:DimensionlessNonAut} when $\Lambda_r(\tau)=\lambda^+$. The fixed point when $\Lambda_r(\tau)=\lambda^+$ must be the origin by the deterministic analysis conducted in Section \ref{DetModel}. By Proposition \ref{Prop: No r-tip}, there will be no rate-induced tipping away from $\mathcal{O}$. If we look at tipping away from $\mathcal{S}^-$, we have to tip to $\mathcal{O}$, independent of $r$, undergoing bifurcation tipping, as we have the annihilation of two fixed points.

\

\noindent The case we will exemplify is when we assume that the values of $c$ and $V_p$ are chosen such that there are three fixed points of Equation \eqref{Eqn:DimensionlessNonAut} when $\Lambda_r(\tau)=\lambda^-: \mathcal{O}, \mathcal{U}^-,\mathcal{S}^-$ and three fixed points of Equation \eqref{Eqn:DimensionlessNonAut} when $\Lambda_r(\tau)=\lambda^+:\mathcal{O}, \mathcal{U}^+,\mathcal{S}^+$ . We will make use of Proposition \ref{Prop_FI} within the proof of Theorem \ref{Thm:Rtip}. See Appendix C for the proof of this proposition.

\begin{proposition}
Suppose $V_p,c>0$ such that $p(v,V_p,c)$ has two positive zeros, $0<v_1<v_2$. If $a=(a_1,a_2)$ and $b=(b_1,b_2)$ satisfy 
$$
\begin{cases} 
v_1<a_1<v_2 \\
\sqrt[3]{\frac{a_1^2}{{(1-\gamma)}(\frac{V_p}{V_p^-})^2+\gamma a_1^2}}<a_2<\frac{a_1}{a_1+c}
\end{cases}
and
\hspace{5mm}
\begin{cases} 
v_2<b_1 \\
\frac{b_1}{b_1+c}<b_2<\sqrt[3]{\frac{b_1^2}{{(1-\gamma)}(\frac{V_p}{V_p^-})^2+\gamma b_1^2}}
\end{cases}$$
then the box $K_{a,b}=[a_1,b_1]\times[a_2,b_2]$ is forward invariant with respect to the flow.
\end{proposition}
\label{Prop_FI}

\begin{theorem}
Assume that $\Lambda_r(\tau),V_p(\Lambda_r(\tau))$, and $c(\Lambda_r(\tau))$ are chosen such that there are three fixed points at $\Lambda_r(\tau)=\lambda^-:\mathcal{O},\mathcal{U}^-,\mathcal{S}^-$ and three fixed points at $\Lambda_r(\tau)=\lambda^+:\mathcal{O},\mathcal{U}^+,\mathcal{S}^+$. Assume that there exist paths $(\tau,p_u(\tau))$ and $(\tau,p_s(\tau))$ and are distinct for all values of $\tau$. If either $V_p(\Lambda_r(\tau))$ or $c(\Lambda_r(\tau))$ is nonincreasing as a function of $\tau$, there can be no rate-induced tipping away from the stable storm state $\mathcal{S}^-$ to the non-storm state $\mathcal{O}$.
\label{Thm:Rtip}
\end{theorem}

\begin{proof}
We will prove that if either $V_p(\Lambda_r(\tau))$ or $c(\Lambda_r(\tau))$ is nonincreasing as a function of $\tau$, then $(\tau,p_s(\tau))$ is forward inflowing stable and hence there can be no rate-induced tipping away from $\mathcal{S}^-$.

First suppose $V_p(\Lambda_r(\tau))$ is nonincreasing. Write $p_s(\tau)=(v_2(\tau),m_2(\tau))$. For each value of $V_p$ there is a unique value of $c$, call it $c^*(V_p)$ for which there is exactly one positive zero of polynomial $p$, in Equation \eqref{poly}. Since $p(v,V_p,c^*(V_p))=0$ if and only if $p(\epsilon v,\epsilon V_p,\epsilon c^*(V_p))=0$ for any $\epsilon>0$, it follows that $c^*(\epsilon V_p)=\epsilon c^*(V_p)$. If we let $v^*(V_p)$ denote the unique zero of $p(v,V_p,c^*(V_p))$ then also $v^*(\epsilon V_p)=\epsilon v^*(V_p)$. In particular, $v^*$ is strictly increasing as a function of $V_p$. If we let $m^*(V_p)=\frac{v^*(V_p)}{v^*(V_p)+c^*(V_p)}$ then $m^*(\epsilon V_p)=m^*(V_p)$ for all $\epsilon >0$, and we can call this common value $m^*$. 

Now we would like to find functions $a_1,a_2:\mathbb{R}\rightarrow \mathbb{R}$ that satisfy
\begin{equation}
\begin{aligned}
\begin{cases} 
v^*(V_p(\Lambda_r(\tau)))&<a_1(\tau)<\hspace{3mm} v_2(\tau) \\
\hspace{10mm} m^*&<a_2(\tau)<\hspace{3mm} \frac{a_1(\tau)}{a_1(\tau)+c(\Lambda_r(\tau))}
\end{cases}
\end{aligned}
\end{equation}
for all $\tau \in \mathbb{R}.$ It is possible to find these functions $a_1,a_2$ by the following argument. For any value of $s$, we may assume $c(\Lambda_r)\in(0,c^*(V_p(\Lambda_r)))$ since we are assuming the paths $(\tau,p_u(\tau))$ and $(\tau,p_s(\tau))$ exist and are distinct. This implies $v*(V_p(\Lambda_r(\tau)))<v_2(\tau)$ and so we can choose $a_1(\tau)$ to satisfy the first inequality. Given this,
\begin{equation}
m^*=\frac{v^*(V_p(\Lambda_r(\tau)))}{v^*(V_p(\Lambda_r(\tau)))+c^*(V_p(\Lambda_r(\tau)))}<\frac{v^*(V_p(\Lambda_r(\tau)))}{v^*(V_p(\Lambda_r(\tau)))+c(\Lambda_r(\tau))}<\frac{a_1(\tau)}{a_1(\tau)+c(\Lambda_r(\tau))},
\end{equation}
and so we can choose $a_2(\tau)$ to satisfy the second inequality. Furthermore, we would like to enforce $a_1,a_2$ be continuous and nonincreasing, which is possible since $v^*(V_p(\Lambda_r(\tau)))$ and $m^*$ are both nonincreasing. 

Also pick constants $b_1,b_2$ such that 
\begin{equation}
\begin{aligned}
\begin{cases} 
\hspace{4mm} v_2(\tau)&<b_1 \\
\frac{b_1}{b_1+c(\Lambda_r(\tau))}&<b_2< \hspace{2mm}\sqrt[3]{\frac{b_1^2}{{(1-\gamma)}(\frac{V_p(\Lambda_r(\tau))}{V_p^-})^2+\gamma b_1^2}}
\end{cases}
\label{EQ:constantb}
\end{aligned}
\end{equation}
for all $\tau$. For each $\tau$ define $K(\tau)=[a_1(\tau),b_1]\times[a_2(\tau),b_2]$. By Proposition \ref{Prop_FI}, each $K(\tau)$ is forward invariant with respect to the flow when $V_p=V_p(\Lambda_r(\tau))$ and $c=c(\Lambda_r(\tau))$. By how we defined $\{K(\tau)\}$, they satisfy Definition \ref{Def:FIS} to show $(\tau,p_s(\tau))$ is forward inflowing stable path and so there can be no rate-induced tipping away from $\mathcal{S}^-$.

Next suppose $c(\Lambda_r(\tau))$ is nonincreasing. For each value of $c$ there is a unique value of $V_p$, call it $V_p^*(c)$ for which there is exactly one positive zero of the polynomial $p$, in Equation \eqref{poly}. Since $p(v,V_p^*(c),c)=0$ if and only if $p(\epsilon v,\epsilon V_p^*(c),\epsilon c)=0$ for any $\epsilon>0$, it follows that $V_p^*(\epsilon c)=\epsilon V_p^*(c).$ If we let $v*(c)$ denote the unique zero of $p(v, V_p*(c),c)$, then also $v^*(\epsilon c)=\epsilon v^*(c)$. In particular, $v^*$ is strictly increasing as a function of $c$. If we let $m^*(c)=\frac{v^*(c)}{v^*(c)+c)}$, then $m^*(\epsilon c)=m^*(c)$ for all $\epsilon>0$, and we can call this common value $m^*$.

Now we would like to find continuous nonincreasing functions $a_1,a_2:\mathbb{R}\rightarrow \mathbb{R}$ that satisfy 
\begin{equation}
\begin{aligned}
\begin{cases} 
v^*(c(\Lambda_r(\tau)))&<a_1(\tau)<\hspace{3mm} v_2(\tau) \\
\hspace{10mm} m^*&<a_2(\tau)<\hspace{3mm} \frac{a_1(\tau)}{a_1(\tau)+c(\Lambda_r(\tau))}
\end{cases}
\end{aligned}
\end{equation}
for all $\tau \in \mathbb{R}.$ The reasons for why this is possible are the same as above. Also, pick constants $b_1,b_2$ to satisfy Equation \eqref{EQ:constantb} for all $\tau$. Then define $K(\tau)=[a_1(\tau),b_1]\times[a_2(\tau),b_2]$. Once again, these sets can be used to show that $(\tau,p_s(\tau))$ is a forward inflowing stable path. Therefore there can be no R-tipping away from $\mathcal{S}^-$.
\end{proof}

Interestingly, if both $V_p(\Lambda_r(\tau))$, and $c(\Lambda_r(\tau))$ are increasing, then Theorem 3.8 allows for the possibility of rate-induced tipping from $\mathcal{S}^-$ to $\mathcal{O}$ for $r>r_c$, as the next example will demonstrate. Using the results above, a mathematically general form of our ramped parameters $V_p$ and $c$ that will result in rate-induced tipping for $r>r_c$ is given by
\begin{equation}
\begin{aligned}
 V_p(\tau)&=V_p^-(1-\Lambda_r(\tau))+V_p^+ \Lambda_r(\tau), \\
c(\tau)&=kV_p(\tau), 
\end{aligned}
\end{equation}
where $k$ is a correlation coefficient between $V_p$ and $c$, and the functions $\Lambda_r, V_p$, and $c$ are chosen such that three fixed points, two stable and one saddle, exist for all time.

\subsection{Example of Rate-Induced Tipping in the Tropical Cyclone Model}
\label{sec: example rate}

As the problem we investigate in this analysis is in $\mathbb{R}^2$, we must understand the basin of attraction in two-dimensional space. For both $\Lambda_r(\tau)=\lambda^-$ and $\Lambda_r(\tau)=\lambda^+$, we have two asymptotically stable fixed points that are separated by a saddle node, whose stable manifold forms a separatrix for the basins of attraction of $\mathcal{O}$ and $\mathcal{S}^-$ and $\mathcal{O}$ and $\mathcal{S}^+$ respectively. If $\mathcal{S}^- \in \mathbb{B}(\mathcal{O},\lambda^+)$, we will have rate-induced tipping away from $\mathcal{S}^-$ by Theorem \ref{THM: rate induced conditions}

Using what we learned in Section \ref{RITTCM}, we choose a parameter shift $\Lambda_r(\tau)$, and increasing functions $V_p$ and $c$ that are dependent on $\Lambda_r(\tau)$ such that $\mathcal{S}^- \in \mathbb{B}(\mathcal{O},\lambda^+)$. One such choice is
\begin{align}
    \displaystyle \Lambda_r(\tau)&=\frac{1}{2}(1+\tanh(r \tau)), \label{ramp_ex}\\
    \displaystyle  V_p(\tau)&=90\Lambda_r(\tau)+10, \label{VP_ex}\\
    \displaystyle c(\tau)&=0.13V_p(\tau), \label{S_ex} 
\end{align}
where the ratio between $V_p(\tau)$ and $c(\tau)$ is fixed at $0.13$ to guarantee that there are always three equilibria. For this choice in functions, $\mathcal{S}^- \in \mathbb{B}(\mathcal{O},\lambda^+)$, as seen in Figure \ref{FIG: basins and rtip}(a). This indicates by Theorem \ref{THM: rate induced conditions} that there will rate-induced tipping away from $\mathcal{S}^-$ to $\mathcal{O}$ for sufficiently large $r > 0$.

Using our approach from Section \ref{sec:quick intro rtip} to convert Equation \eqref{Eqn:DimensionlessNonAut} back to an autonomous system, we have the system of first order equations given by
\begin{equation}
\begin{aligned}
    \frac{dv}{d \tau} &= \frac{(1-\gamma) V_p(s)^2}{{V_p^-}^2} m^3-(1-\gamma m^3)v^2, \\
    \frac{dm}{d \tau} &= (1-m)v-c(s)m, \\ \frac{ds}{d\tau}&=1 .
\end{aligned}
\label{Eqn:DimensionlessAut_rtip}
\end{equation}

\noindent Solving this system, we determine for $r<r_c$ we endpoint track the stable path from $\mathcal{S}^-$ to $\mathcal{S}^+$. However, when $r>r_c$ we tip from $\mathcal{S}^-$ to $\mathcal{O}$. For $r=r_c$ we have a heteroclinic connection between $\mathcal{S}^-$ and $\mathcal{U}^+$. Via numerical simulations we find that $r_c \in (0.0506279,0.050628)$. We show numerical results of tipping in $(v,m)$ space in Figure \ref{FIG: basins and rtip}(b).

In conclusion, we see that due to the stability of the non-storm state, $\mathcal{O}$, we cannot form a tropical cyclone, but a storm can destabilize with rapidly increasing max potential velocity and wind shear.

\begin{figure}[t!]
\centering
\begin{subfigure}[b]{0.45\textwidth}
    \centering
    \includegraphics[width=.95\textwidth]{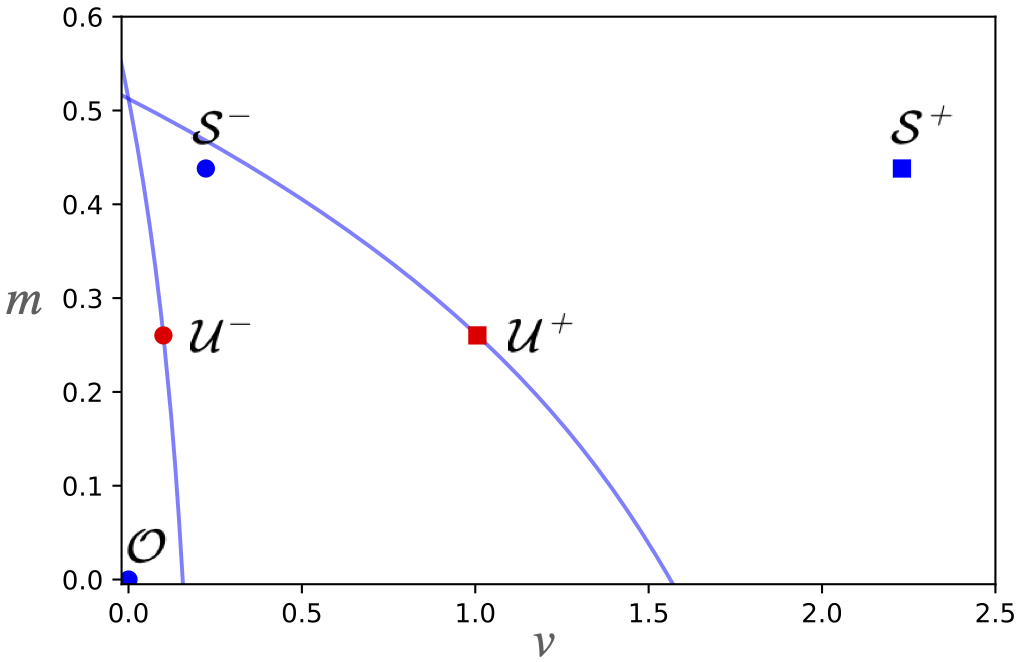}
\end{subfigure}
\begin{subfigure}[b]{0.45\textwidth}
    \centering
    \includegraphics[width=.95\textwidth]{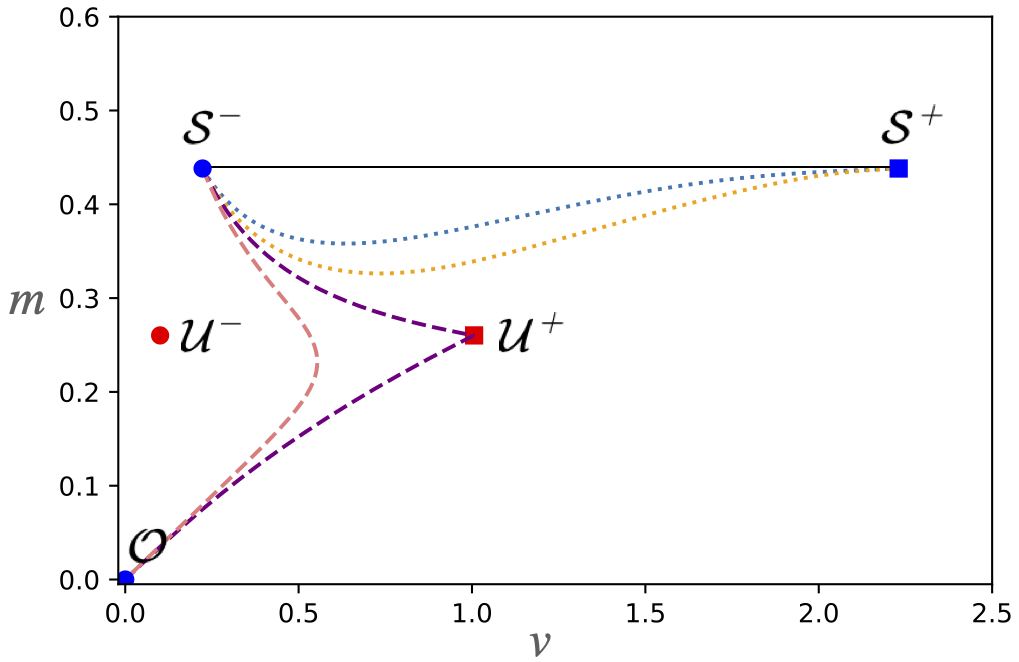}
\end{subfigure}
    \caption{In both plots the fixed points of the system given by Equation \eqref{Eqn:DimensionlessAut} at the start and end of the ramp function, Equation \eqref{ramp_ex}, are shown. $V_p$ and $c$, defined by Equations \eqref{VP_ex} and \eqref{S_ex}, are time dependent and $\gamma=0.43$. The three fixed points at the start of the ramp, $\mathcal{O}, \mathcal{U}^-, \mathcal{S}^-$, correspond to the non-storm state, the unstable storm state, and the stable storm state. These stable fixed points are denoted by blue circles and the saddle node is denoted by a red circle. The three fixed points at the end of the ramp, $\mathcal{O}, \mathcal{U}^+, \mathcal{S}^+$, correspond to the non-storm state, the unstable storm state, and the stable storm state. These stable fixed points are denoted by blue squares and the saddle is denoted by a red square. (a) Blue curves correspond to the stable manifolds of $\mathcal{U}^-$ and $\mathcal{U}^+$. (b) Plot of the solutions to the system given by Equation \eqref{Eqn:DimensionlessAut} for different values or $r$. The black curve is the solution when $r=0$ and end-point tracks the stable path from $\mathcal{S}^-$ to $\mathcal{S}^+$. The blue (orange) dotted curve corresponds to the solution for $r=0.03$ ($r=0.04$), which end-point tracks the stable path from $\mathcal{S}^-$ to $\mathcal{S}^+$. The purple (pink)  dashed curve corresponds to the solution for $r=0.050628$ ($r=0.08$), which do not endpoint track the stable path, and tip from $\mathcal{S}^-$ to $\mathcal{O}$.}
    \label{FIG: basins and rtip}
\end{figure}

\section{Noise-Induced Tipping in the Tropical Cyclone Model} \label{Noise}
In this section we study noise-induced transitions between the stable states $\mathcal{O}$ and $\mathcal{S}$ for the stochastic differential equation 
\begin{equation}
\begin{aligned}
    dv &=f(v,m) d\tau+ \sigma_1 dW_1, \\
    dm &=g(v,m) d\tau + \sigma_2 dW_2, \label{Eq:Noise:SDE}
    \end{aligned}
\end{equation}
where $\sigma_1,\sigma_2>0$, $W_1, W_2$ are independent Brownian motions, $f,g$ are defined as in Equation \eqref{Eqn:DimensionlessAut} and in this section we are returning to the dimensionless coordinates introduced in Section \ref{DetModel}. In Section \ref{DetModel} we showed that for small wind shear the separation between $\mathcal{O}$ and $\mathcal{U}$ is relatively small in comparison with the separation between $\mathcal{U}$ and $\mathcal{S}$. Consequently, we expect $\mathcal{O}$ is highly susceptible to noise-induced tipping while $\mathcal{S}$ is more robust to random fluctuations. Moreover, the existence of a one-dimensional center manifold near $\mathcal{O}$ indicates that the deterministic flow is comparatively weak when restricted to this manifold, providing a natural region in phase space that is susceptible to noise-induced transitions. 

Recall from Section \ref{DetModel} that for physical reasons $v,m\geq 0$  and additionally it can be shown that the autonomous system extended to $\mathbb{R}^2$ is unstable for $v<0$ but the first quadrant is invariant. However, since realizations of Equation \eqref{Eq:Noise:SDE} can enter these nonphysical regions of phase space, we interpret Equation \eqref{Eq:Noise:SDE} to have reflecting boundary conditions along the lines $v=0$ and $m=0$. That is, for $(\hat{v},\hat{m})\in \mathbb{R}^2$ we consider the system
\begin{equation}
\begin{aligned}
    d\hat{v} &=\hat{f}(\hat{v},\hat{m}) d \tau+ \sigma_1 dW_1, \\
    d\hat{m} &=\hat{g}(\hat{v},\hat{m}) d \tau + \sigma_2 dW_2, \label{Eq:Noise:Reflecting}
\end{aligned}
\end{equation}
where the reflected components of the vector field are defined by
\begin{equation}
\begin{aligned}
\hat{f}(\hat{v},\hat{m})&=\begin{cases} 
      f(\hat{v},\hat{m}) & \hat{v}>0, \hat{m}>0 \\
      -f(-\hat{v},\hat{m}) & \hat{v}<0, \hat{m}>0 \\
      -f(-\hat{v},-\hat{m}) & \hat{v}<0, \hat{m}<0 \\
      f(\hat{v},-\hat{m}) & \hat{v}>0, \hat{m}<0
\end{cases}, \\
\hat{g}(\hat{v},\hat{m})&=\begin{cases}=g(\hat{v},\hat{m}) & \hat{v}>0, \hat{m}>0 \\
      -g(-\hat{v},\hat{m}) & \hat{v}<0, \hat{m}>0 \\
      -g(-\hat{v},-\hat{m}) & \hat{v}<0, \hat{m}<0 \\
      g(\hat{v},-\hat{m}) & v>0, m<0
\end{cases}.
\end{aligned} \label{Eq:Noise:ReflectedComponents}
\end{equation}
Realizations to Equation \eqref{Eq:Noise:Reflecting} are then mapped to realizations of Equation \eqref{Eq:Noise:SDE} with reflecting boundary conditions by setting $(v(\tau),m(\tau))=(|\hat{v}(\tau)|,|\hat{m}(\tau)|)$; see Figure \ref{Fig:Noise:reflecting}(a-b). Throughout the rest of this document we will suppress this notation with the understanding that when referring to $f,g$ we are in fact using the reflected components $\hat{f},\hat{g}$ and when referring to Equation \eqref{Eq:Noise:SDE} we are in fact referring to Equation \eqref{Eq:Noise:Reflecting}.

To be precise when discussing noise-induced tipping, we provide the following definition. 
\begin{definition} 
A noise-induced transition from $\mathcal{O}$ to $\mathcal{S}$, or noise-induced tipping event from $\mathcal{O}$ to $\mathcal{S}$,  is a realization of Equation \eqref{Eq:Noise:SDE} satisfying $(v(0),m(0))=(0,0)$, and there exists $\tau^*\in \mathbb{R}^+$ for which $(v(\tau^*), m(\tau^*))\in \overline{\mathbb{B}(\mathcal{S})}$ and for $\tau<\tau^*$, $(v(\tau),m(\tau))\in \mathbb{B}(\mathcal{O})$. The variable $\tau^*$ is itself a random variable, specifically a stopping time for this process, and is referred to as the tipping time from $\mathcal{O}$ to $\mathcal{S}$.
\end{definition}
We note that similar definition holds for noise-induced transitions from $\mathcal{S}$ to $\mathcal{O}$ and the corresponding tipping time $\tau^*$ from $\mathcal{S}$ to $\mathcal{O}$. Moreover, since the noise is additive, it follows that for systems $\mathbb{P}(\tau^*<\infty)=1$.

In Figure \ref{Fig:Noise:reflecting}(c-d) and Figure \ref{Fig:Noise:reflecting}(e) we plot the time series of $m(\tau)$ for realizations of Equation \eqref{Eq:Noise:SDE} that start start at $\mathcal{O}$ and $\mathcal{S}$ respectively. From these numerical experiments we can obtain further evidence that $\mathcal{O}$ is far more susceptible to noise-induced tipping than $\mathcal{S}$. That is, the expected value of the tipping time from $\mathcal{O}$ to $\mathcal{S}$ is dramatically smaller than from $\mathcal{S}$ to $\mathcal{O}$. Moreover, as seen in Figure \ref{Fig:Noise:reflecting}, the noise-induced tipping events from $\mathcal{O}$ to $\mathcal{S}$ appear to be concentrated about a particular region in phase space. To validate these numerical observations, we will use the Freidlin-Wentzell theory of large deviations to quantify the most probable noise-induced transitions as well as the expected tipping time. 

\begin{figure}[ht!]
\centering
\begin{subfigure}[b]{0.4\textwidth}
    \centering
    \includegraphics[width=1\textwidth]{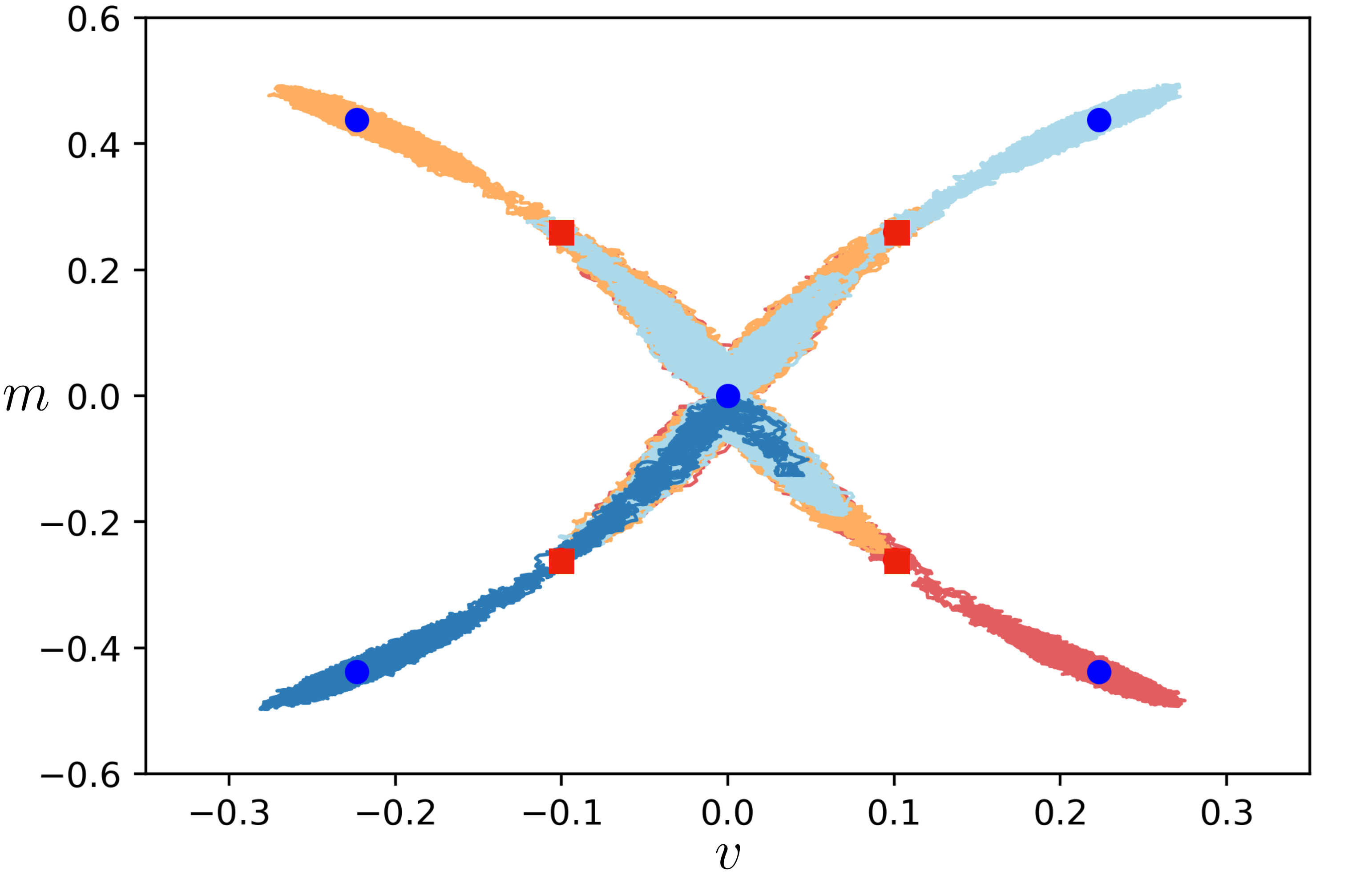}
    \caption{}
\end{subfigure}
\begin{subfigure}[b]{0.4\textwidth}
    \centering
    \includegraphics[width=1\textwidth]{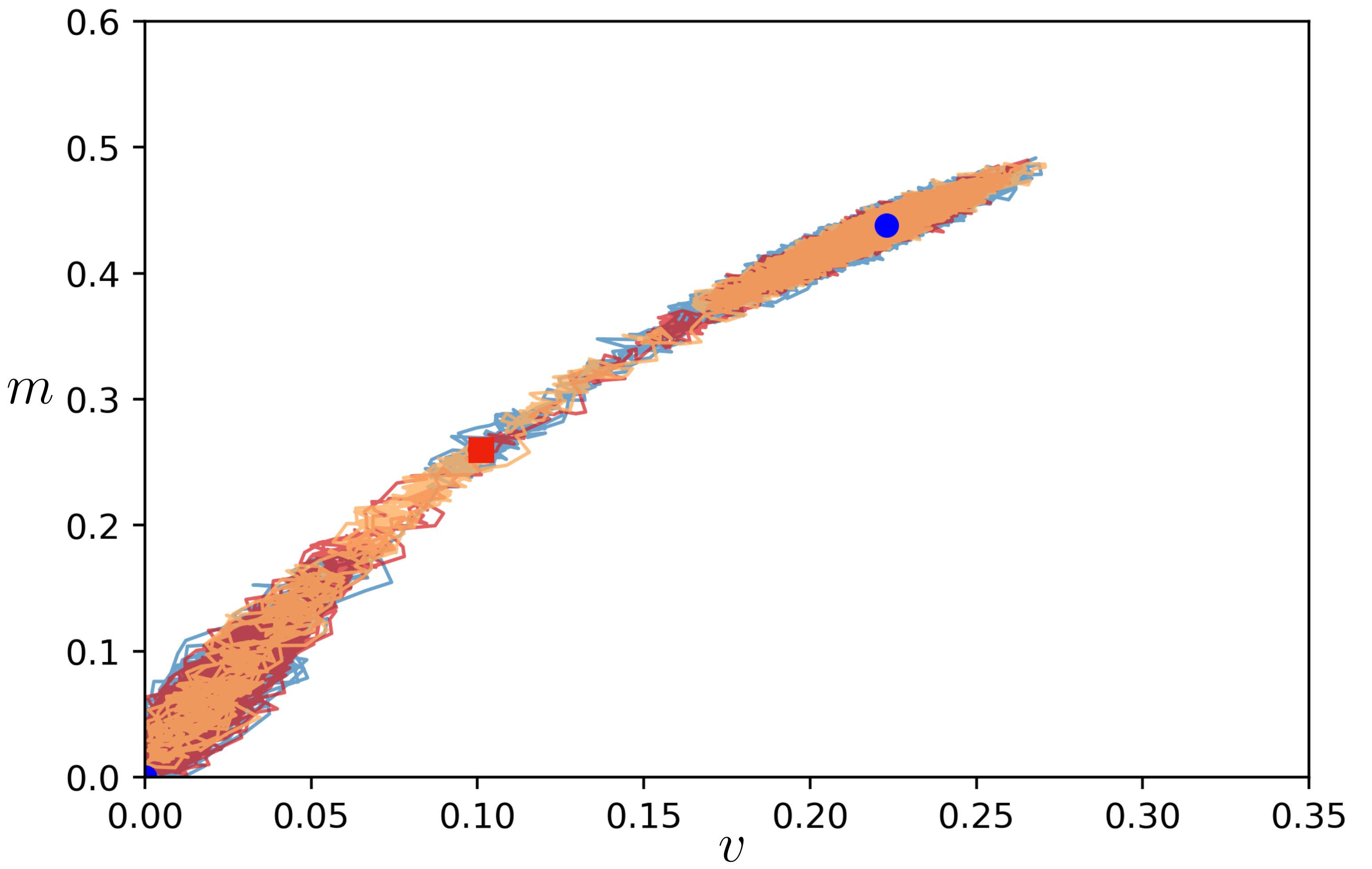}
    \caption{}
\end{subfigure}
    \begin{subfigure}[b]{0.4\textwidth}
    \centering
    \includegraphics[width=1\textwidth]{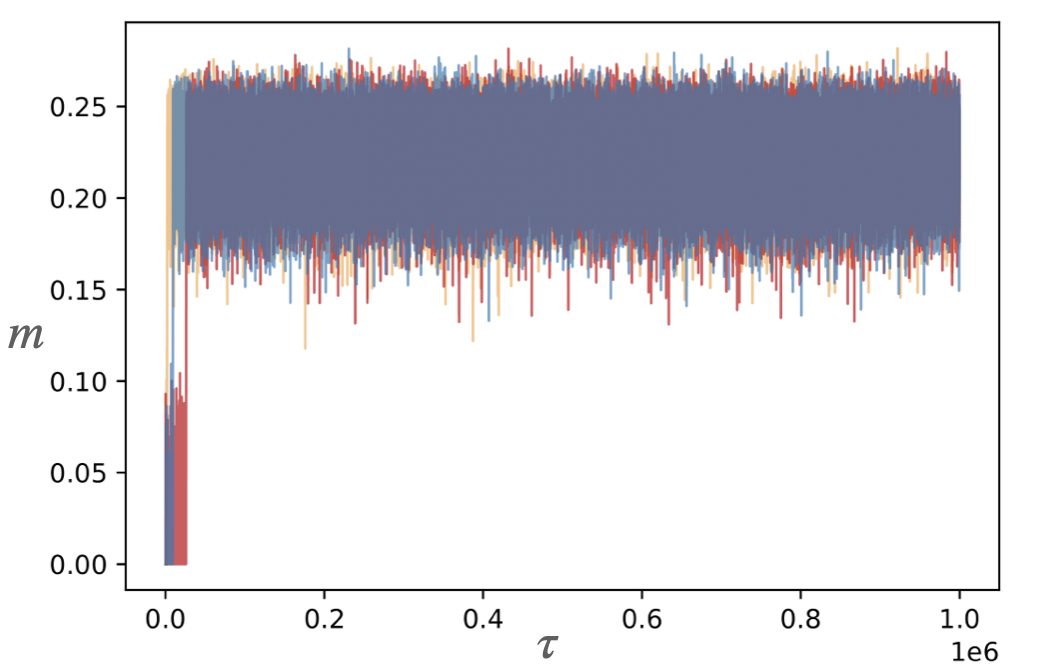}
    \caption{}
\end{subfigure}
\begin{subfigure}[b]{0.4\textwidth}
    \centering
    \includegraphics[width=1\textwidth]
    {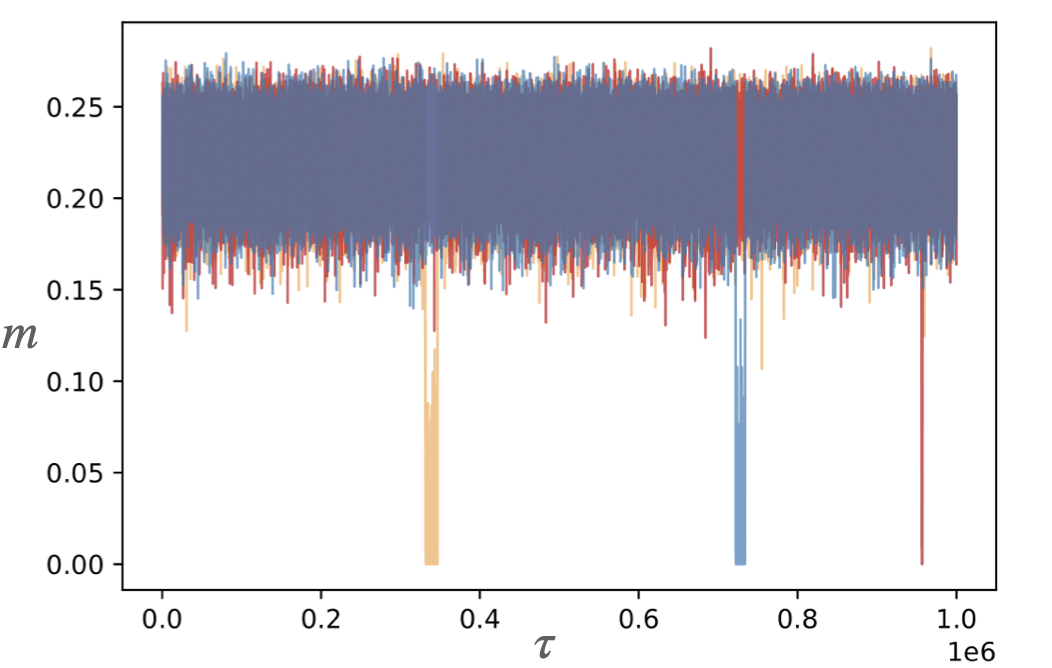}
    \caption{}
\end{subfigure}
\begin{subfigure}[b]{0.4\textwidth}
    \centering
    \includegraphics[width=1\textwidth]
    {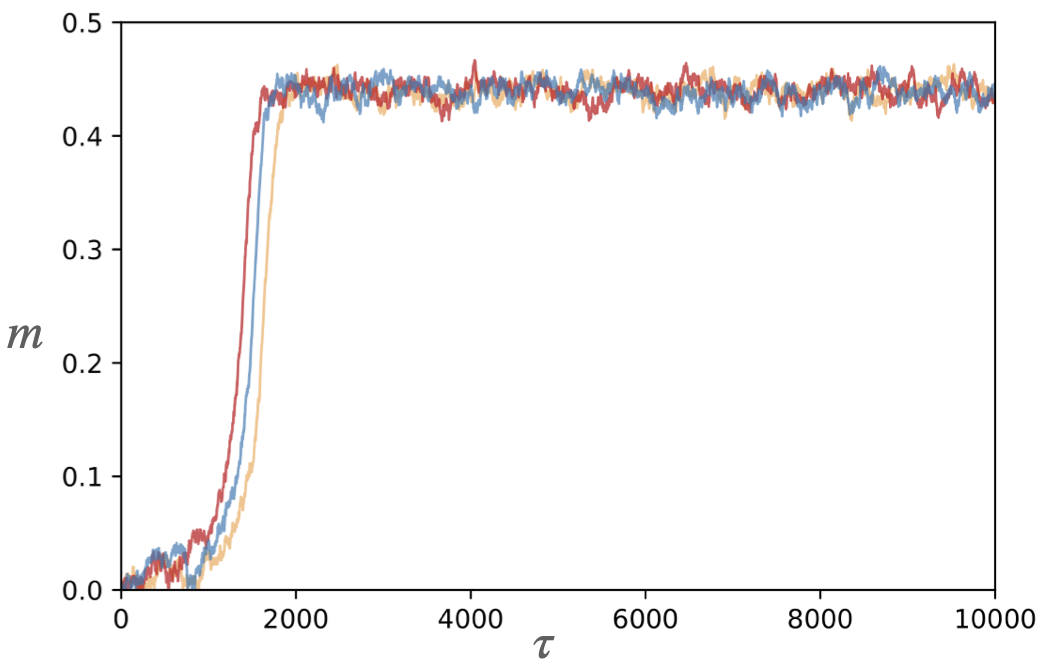}
    \caption{}
\end{subfigure}

    \caption{(a) Realizations of Equation \eqref{Eq:Noise:Reflecting} with $\sigma_1=\sigma_2=0.005, \gamma=0.43$, $V_p=10$, and $c=0.286$. (b) Realizations of Equation \eqref{Eq:Noise:SDE} with reflecting boundary conditions, computed by mapping realizations of Equation \eqref{Eq:Noise:Reflecting} to the first quadrant. In both (a-b), stable fixed points for the deterministic dynamics correspond to blue circles while red squares correspond to the saddles. (c-e) Time series of three realizations of Equation \eqref{Eq:Noise:SDE} with differing values of $c$, with $d \tau=0.1$, and $\tau_f=10^6$. (c) $c=0.286$ with realizations initialized at the non-storm state $\mathcal{O}$. (d) $c=0.286$ with realizations initialized at the stable storm state $\mathcal{S}$. (e) $c=0.22$ with realizations initialized at the non-storm $\mathcal{O}$.}    
    \label{Fig:Noise:reflecting}
\end{figure}

\subsection{A Quick Introduction to Most Probable Transitions}
In this subsection we present the Freidlin-Wentzell theory of large deviations to provide a framework for computing most probable transition paths from $\mathcal{O}$ to $\mathcal{S}$ noting that the same theory can be applied to compute most probable transition paths from $\mathcal{S}$ to $\mathcal{O}$. This framework is presented in Freidlin and Wentzel's book \cite{freidlin2012random},  the review article by Forgoston and Moore \cite{forgoston2018primer}, and the review article by Berglund \cite{berglund2013kramers} (for gradient systems). In particular, reference \cite{galfi2021applications} provides a nice introduction to the Freidlin-Wentzell theory of large deviations within the context of a climate application. To simplify the following exposition, we let $F=(f,g)$ denote the vector field with components $f$ and $g$, introduce the matrix
\begin{equation}
\Sigma=\begin{bmatrix}
\sigma_1^{-2} & 0\\
0 & \sigma_2^{-2}
\end{bmatrix},
\end{equation}
and define for $\mathbf{v}_1,\mathbf{v}_2\in \mathbb{R}^2$ the weighted inner product $\langle \mathbf{v}_1,\mathbf{v}_2\rangle_{\Sigma}=\mathbf{v}_1^{T}\Sigma \mathbf{v}_2$ and the weighted norm $\|\mathbf{v}_2-\mathbf{v}_1\|_{\Sigma}^2=\langle \mathbf{v}_2-\mathbf{v}_1\rangle_{\Sigma}$. We begin with a definition of a most probable path that summarizes and combines the definitions appearing in \cite{heymann2008geometric, freidlin2012random, berglund2013kramers,galfi2021applications}.

\begin{definition}A curve $\Psi(s)=(\psi_1(s),\psi_2(s))$ is a \emph{most probable transition path between $\mathcal{O}$ and $\mathcal{S}$ on the domain $[\tau_0,\tau_f]$} if it minimizes the Freidlin-Wentzell rate functional 
\begin{equation}
    I[\Psi]=\frac{1}{2}\int_{\tau_0}^{\tau_f}\|\dot{\Psi}-F(\Psi)\|_{\Sigma}^2ds,
\end{equation}
over the admissible set
\begin{equation}
\mathcal{A}_{\mathcal{O}}^{(\tau_0,\tau_f)}=\{\Psi\in H^1([\tau_0,\tau_f];\mathbb{R}^2):\Psi(\tau_0)=\mathcal{O} \text{ and } \Psi(\tau_f)=\mathcal{S}\}.
\end{equation}
The \emph{most probable path} $\Psi^*$, if it exists, is the minimizer of the double optimization problem
\begin{equation}
\inf_{[\tau_0,\tau_f]}\inf_{\Psi \in \mathcal{A}_{O}^{(\tau_0,\tau_f)}}I[\Psi]. \label{Eqn:DoubleOpt}
\end{equation}
Note, in terms of notation we represent the most probable transition path as $\Psi=(\psi_1(s),\psi_2(s))$ instead of $(v(s),m(s))$ to distinguish it from a generic realization or the deterministic dynamics. Additionally, note the functional $I$ as defined above also depends on $\tau_0,\tau_f$ but we suppress this dependence to simplify notation. Moreover, we will later show that the minimizer over this double optimization can only be obtained when $\tau_0=-\infty$ and $\tau_f=\infty$. 
\end{definition}

Summarizing the key concepts in \cite{freidlin2012random}, the Freidlin-Wentzell large deviations principle for Equation \eqref{Eq:Noise:SDE} states formally that as  $\sigma_1,\sigma_2\rightarrow 0$ the probability that a realization $(v(\tau),m(\tau))$ of Equation \eqref{Eq:Noise:SDE} remains within a $\delta>0$ neighborhood of $\Psi\in \mathcal{A}_{\mathcal{O}}^{(\tau_0,\tau_f)}$ is given by
\begin{equation}\label{eqn:ldp}
\mathbb{P}\left(\sup_{\tau\in [\tau_0,\tau_f]} \|((v(\tau),m(\tau))-\Psi(\tau)\| <\delta\right) \asymp e^{-I[\Psi]},
\end{equation}
where $\asymp$ denotes logarithm equivalence\footnote{For real sequences $x_{\varepsilon}$, $y_{\varepsilon}$ we say $x_{\varepsilon}\asymp y_{\varepsilon}$ if $\lim_{\varepsilon\rightarrow 0}\frac{\ln(x_{\varepsilon})}{\ln(y_{\varepsilon})}=1$.}. Consequently, in the limit $\sigma_1,\sigma_2\rightarrow 0$, the most probable path $\Psi^*$ can be interpreted as the mode of the probability distribution on $\mathcal{A}_{\mathcal{O}}^{(\tau_0,\tau_f)}$. Additionally, the expected value of the tipping time can be computed from knowledge of the most probable path by the formula
\begin{equation}
\mathbb{E}[\tau^*]\asymp e^{I[\Psi^*]}.\label{Eqn:LDP:Expect_Value}
\end{equation}
 Heuristically, Equation \eqref{Eqn:LDP:Expect_Value} can be justified by letting $p=\exp(-I[\Psi^*])$ approximate the probability of a realization of Equation \eqref{Eq:Noise:SDE} leaving $\mathbb{B}(\mathcal{O})$ in an interval of time $[\tau_0,\tau_f]$.  In the limit $\sigma_1,\sigma_2\rightarrow 0$ it can be shown that $p$ has approximately a geometric distribution, i.e., the realization either leaves $\mathbb{B}(\mathcal{O})$ or returns to $\mathcal{O}$ in the given interval of time, and thus following this logic $\mathbb{E}[\tau^*]\approx 1/p=\exp(I[\Psi^*])$ \cite{berglund2013kramers}.

Equation \eqref{eqn:ldp} indicates that the most probable path $\Psi^*$ defined above corresponds to the curve in phase space in which noise-induced transitions from $\mathcal{O}$ to $\mathcal{S}$ concentrate about in the vanishing noise limit $\sigma_1,\sigma_2\rightarrow 0$. Moreover, the infimum over $[\tau_0,\tau_f]$ can be interpreted as resulting from accounting for all possible parameterizations of the curve. However, note that $I$ vanishes along curves in which $\Psi$ tracks the deterministic dynamics, i.e., $\dot{\Psi}=F(\Psi)$. Consequently, once $\Psi^*$ crosses the separatrix $\partial \mathbb{B}(\mathcal{O})\cap \partial \mathbb{B}(\mathcal{S})$, the most probable path $\Psi^*$ will simply satisfy $\dot{\Psi^*}=F(\Psi^*)$ in this region of phase space. Therefore, we can consider the equivalent optimization problem
\begin{equation}
\inf_{[\tau_0,\tau_f]}\inf_{\Psi \in \overline{\mathcal{A}}_{O}^{(\tau_0,\tau_f)}}I[\Psi], \label{Eqn:DoubleOptInt}
\end{equation}
where 
\begin{equation}
\overline{\mathcal{A}}_{O}^{(\tau_0,\tau_f)}=\{\Psi\in H^1([\tau_0,\tau_f];\mathbb{R}^2):\Psi(\tau_0)=\mathcal{O} \text{ and } \Psi(\tau_f)\in \partial \mathbb{B}(\mathcal{O})\cap \partial \mathbb{B}(\mathcal{S})\}.
\end{equation}

We can reduce the complexity of this optimization problem by proving that we can take $\tau_f=\infty$ in Equation \eqref{Eqn:DoubleOptInt} by studying the corresponding Euler-Lagrange equations and its natural boundary conditions. Taking the first variation of $I$ over $\mathcal{A}_{\mathcal{O}}^*$ and integrating by parts yields
\begin{equation}
\delta I
=\int_{\tau_0}^{\tau_f}\left\langle -\ddot{\Psi}+\nabla F(\Psi)\dot{\Psi}-\Sigma^{-1}\nabla F^T(\Psi)\Sigma(\dot{\Psi}-F(\Psi)),\delta \Psi\right\rangle_{\Sigma}ds+\left.\left\langle\dot{\Psi}-F(\Psi),\delta \Psi\right\rangle_{\Sigma}\right|_{\tau_f}.
\end{equation}
Consequently, the Euler-Lagrange equations are given by
\begin{equation}
\ddot{\Psi}=\nabla F(\Psi)\dot{\Psi}-\Sigma^{-1}\nabla F^T(\Psi)\Sigma(\dot{\Psi}-F(\Psi)), \label{Eqn:Euler-Lagrange}
\end{equation}
with the additional ``natural boundary condition'' that $\dot{\Psi}(\tau_f)=F(\Psi(\tau_f))$ on the separatrix, i.e., $\Psi$ tracks the flow on the separatrix. Therefore, since the separatrix in this problem corresponds to the stable manifold of $\mathcal{U}$ and $\Psi$ can track the flow of $F$ at zero cost, it follows that the infimum of \eqref{Eqn:DoubleOptInt} is obtained when $\tau_f=\infty$, $\Psi^*$ terminates at $\mathcal{S}$, and $\lim_{s\rightarrow \infty}\dot{\Psi}^*(s)=F(\mathcal{S})=0$. 

We now show that we can further reduce the complexity of this problem by assuming $\tau_0=-\infty$. We do this by by putting Equation \eqref{Eqn:Euler-Lagrange} into Hamiltonian form through the Legendre transform $\mathbf{p}=(p_1,p_2)=\Sigma (\dot{\Psi}-F(\Psi))$ \cite{forgoston2018primer}. Through this transformation, we obtain the following Hamiltonian system
\begin{equation}
\begin{aligned}
    \dot{\Psi}&=F(\Psi)+\Sigma^{-1} \mathbf{p},\\
    \dot{\mathbf{p}}&=-\nabla F^T\mathbf{p}, \label{Eqn:Hamiltonian} 
    \end{aligned} 
\end{equation}
with corresponding Hamiltonian
\begin{equation}
H=\frac{1}{2}\|\mathbf{p}\|_{\Sigma^{-1}}^2+\langle F(\Psi),\mathbf{p}\rangle. \label{Eqn:HamiltonianScalar}
\end{equation}
With this change of variables, the Freidlin-Wentzell rate functional transforms into the following simple form:
\begin{equation}
    I[\Psi,\mathbf{p}]=\frac{1}{2}\int_{\tau_0}^{\tau_f}\|\mathbf{p}\|_{\Sigma}^2ds. \label{Eqn:Noise:FunctionalHamiltonian}
\end{equation}
Since the Hamiltonian is conserved along the flow generated by \eqref{Eqn:Hamiltonian} and the most probable path satisfies $\lim_{s\rightarrow \infty}\dot{\Psi}^*(s)=F(\mathcal{S})=0$, it follows immediately that $H=0$ on the most probable path $\Psi^*$. Moreover, since $F(\mathcal{O})=0$ it follows that for the conjugate momentum $\mathbf{p}^*$ corresponding to $\Psi^*$, $\lim_{\tau \rightarrow \tau_0}\mathbf{p}^*(\tau)=0$. That is, $(\Psi^*,\mathbf{p}^*)$, if it exists, is a heteroclinic connection between $(\mathcal{O},0)$ and $(\mathcal{U},0)$ in this Hamiltonian system and thus $\tau_0=-\infty$. 

There are some additional properties of Equation \eqref{Eqn:Hamiltonian} which will aid our later analysis. 
\begin{enumerate}
\item Equation \eqref{Eqn:Hamiltonian} contains an invariant submanifold defined by $\mathbf{p}=0$ on which the system follows the deterministic dynamics $\dot{\Psi}=F(\Psi)$.
\item The fixed points of Equation \eqref{Eqn:Hamiltonian} retains the deterministic fixed points with zero conjugate momentum: $(\Psi,\mathbf{p})=(\mathcal{O},0)$, $(\mathcal{U},0)$, $(\mathcal{S},0)$. 
\item The Jacobian of Equation \eqref{Eqn:Hamiltonian} at the above fixed points is of the form
\begin{equation}
    J(\cdot,0)=\begin{bmatrix}
\nabla F(\cdot) & \Sigma^{-1} \\
0 & -\nabla F^T(\cdot) 
    \end{bmatrix}.
\end{equation}
\end{enumerate}
It follows from these properties that if we let $\lambda_1,\lambda_2$ denote the eigenvalues of $\nabla F(\cdot)$  then $\lambda_1,\lambda_2,-\lambda_1,-\lambda_2$ are eigenvalues of $J(\cdot,0)$. Thus, for every stable (unstable) manifold at $\mathcal{O}$ of the deterministic dynamics there is a corresponding unstable (stable) manifold  

Finally, we conclude this brief overview of the Freidlin-Wentzell theory with a discussion of the numerical technique we use to compute most probable transition paths. Since Equation \eqref{Eq:Noise:SDE} is a low dimensional system with a simple set of fixed points, we will numerically solve the boundary value problem given by Equation \eqref{Eqn:Euler-Lagrange} with the boundary conditions $\Psi(\tau_0)=\mathcal{O}$ and $\Psi(\tau_f)=\mathcal{U}$ by computing steady states of the corresponding gradient flow. Specifically, we introduce an artificial time $s$ and consider the evolution equation $\partial_s \Psi=-\frac{\delta I}{\delta \Psi}$ with Dirchlet boundary conditions:
\begin{equation}
\begin{aligned}
&\frac{\partial \Psi}{\partial s}=\frac{\partial^2 \Psi}{\partial \tau^2}-\nabla F(\Psi)\frac{\partial \Psi}{\partial \tau}+\Sigma^{-1}\nabla F^T(\Psi)\Sigma\left(\frac{\partial \Psi}{\partial \tau}-F(\Psi)\right),\\
&\Psi(s,\tau_0)=\mathcal{O} \text{ and }\Psi(s,\tau_f)=\mathcal{S}.
\end{aligned}\label{Noise:Eqn:GradFlow}
\end{equation}
The rate functional $I$ acts as a Lyapunov functional in the sense that solutions of Equation \eqref{Noise:Eqn:GradFlow} satisfy $\frac{d}{ds}I[\Psi(s,\tau)]\leq 0$ and $\frac{d}{ds}I[\Psi(s,\tau)]=0$ if and only if $\Psi(s,\tau)$ solves the Euler-Lagrange equations \eqref{Eqn:Euler-Lagrange}. Consequently, most probable transition paths $\Psi^*(\tau)$ can be computed as the stationary solutions of Equation \eqref{Noise:Eqn:GradFlow}, i.e., $\lim_{s\rightarrow \infty}\Psi(s,\tau)=\Psi^*(\tau)$. 


\subsection{Most Probable Transition Paths for the Tropical Cyclone Model}
We now apply the above framework to quantify the susceptibility of $\mathcal{O}$ to noise-induced tipping. In Figure \ref{r} we plot numerical approximations of the most probable transition paths computed as the stationary states of Equation \eqref{Noise:Eqn:GradFlow}. From this figure we do indeed see that the most probable transition path from $\mathcal{O}$ to $\mathcal{S}$ remains close to the center manifold near $\mathcal{O}$. In this subsection we will validate this claim by finding an explicit formula for an approximation of the most probable path near the origin and use Equation \eqref{Eqn:LDP:Expect_Value} to compute a scaling law for the expected tipping time. Specifically, we will use the Hamiltonian formulation to approximate candidates for the heteroclinic orbit that exits $\mathcal{O}$ and terminates at $\mathcal{S}$. 

\begin{figure}[t!]
\centering
\begin{subfigure}[b]{0.4\textwidth}
    \centering
    \includegraphics[scale=.35]{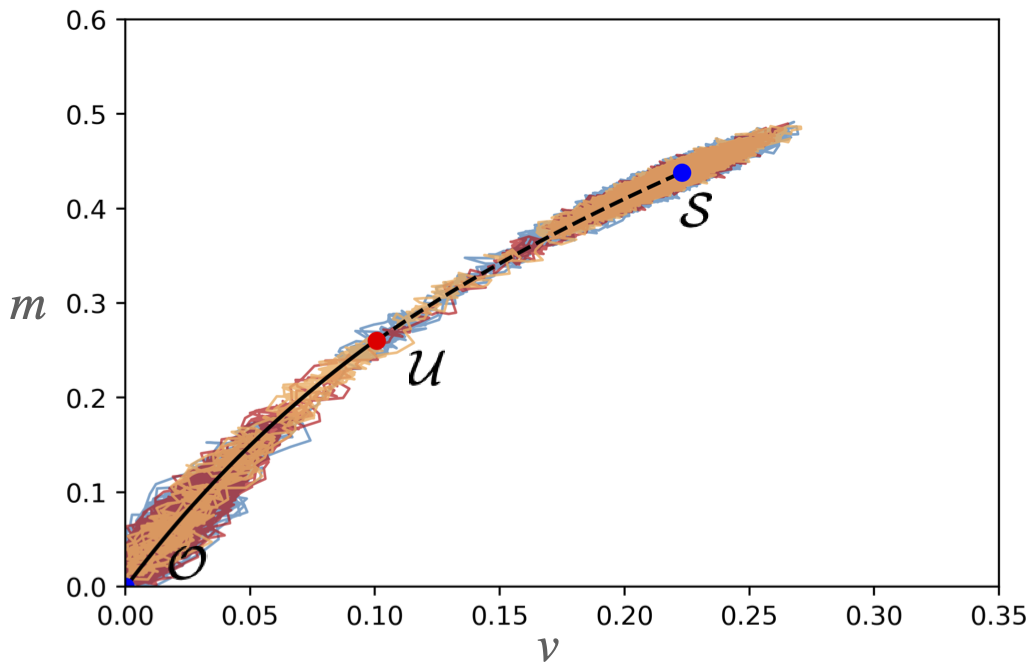}
    \caption{}
\end{subfigure}
\begin{subfigure}[b]{0.4\textwidth}
    \centering
    \includegraphics[scale=.35]{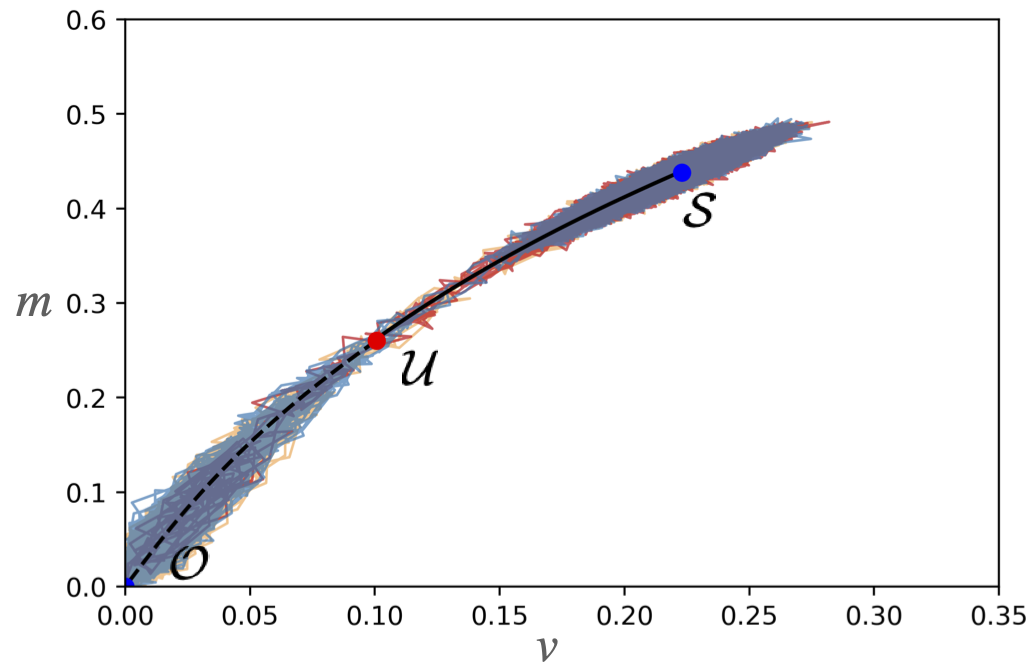}
    \caption{}
\end{subfigure}
\begin{subfigure}[b]{0.4\textwidth}
    \centering
    \includegraphics[scale=.35]{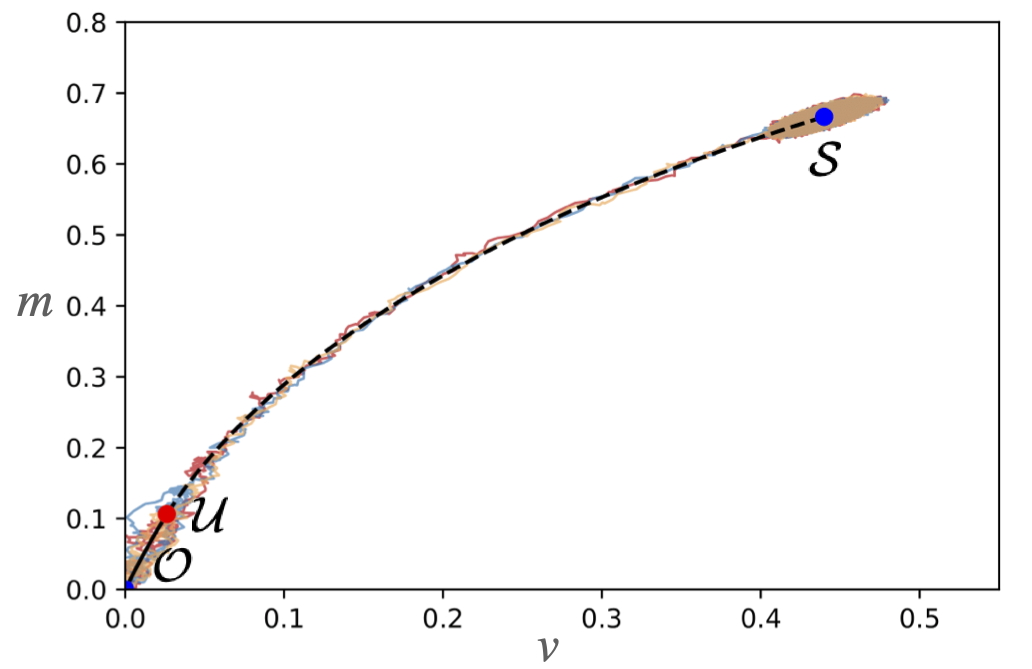}
    \caption{}
\end{subfigure}
    \caption{Plots of the most probable path using a combination of the gradient flow and the deterministic dynamics, overlaid on realizations of Equation \eqref{Eq:Noise:SDE} generated with the Euler-Maruyama method with $\tau_f=10^6$ and  $d \tau=0.1$. The blue circles correspond to the stable fixed points, $\mathcal{O},\mathcal{S}$ and the red square corresponds to the saddle, $\mathcal{U}$. The solid black curve represents the piece of the most probable path from the gradient flow and the dashed black curve represents the piece of the most probable path coming from the deterministic dynamics. Parameter values are set as $\sigma_1=\sigma_2=.005, \gamma=0.43$, $V_p=10$, and $c$ varies per plot.
    (a) $c=0.286$. The most probable path from the non-storm state, $\mathcal{O}$, to the stable storm state, $\mathcal{S}$, overlaid on the tipped realizations.
    (b) $c=0.286$. The most probable path from the stable storm state, $\mathcal{S}$, to the non-storm state, $\mathcal{O}$, overlaid on tipped realizations.
    (c) $c=0.22$. The most probable path from the non-storm state, $\mathcal{O}$, to the stable storm state, $\mathcal{S}$, overlaid on the tipped realizations.}
    \label{r}
\end{figure}
First, we note that the Jacobian of Equation \eqref{Eqn:Hamiltonian} at $\mathcal{O}$ is given by
\begin{equation}
\nabla F (\mathcal{O})=\begin{bmatrix}
    0 & 0 \\
    1 & -c
\end{bmatrix}
\end{equation}
and therefore the eigenvalues of $J(\mathcal{O},0)$ are $\pm c$ and $0$ with $0$ having algebraic multiplicity two and geometric multiplicity one. Consequently, at $(\mathcal{O},0)$ there is a one-dimensional unstable manifold $\mathcal{W}^U$, a one-dimensional stable manifold $\mathcal{W}^S$ which overlaps with the stable manifold for the deterministic dynamics, and a two-dimensional center manifold $\mathcal{W}^C$. Consequently, natural candidates for a heteroclinic orbit lie in $\mathcal{W}^U$ or $\mathcal{W}^C$. However, we numerically found that $\mathcal{W}^U$ does not not intersect the stable manifold of $\mathcal{U}$ and thus we focus on trajectories in $\mathcal{W}^C$.

To begin computing the dynamics on $\mathcal{W}^C$ we perform a standard center manifold reduction. That is, we assume that $(\mathcal{O},0)$ can be locally parameterized as the graph of $(\psi_1,p_1)$, i.e., $\mathcal{W}^C=(\psi_1,\psi_2(\psi_1,p_1),p_1,p_2(\psi_1,p_1))$ where  $\psi_2(\psi_1,p_1)$, $p_2(\psi_1,p_1)$ are analytic functions with power series of the form
\begin{equation}
\begin{aligned}
\psi_2(\psi_1,p_1)&=\sum_{i=1}^{\infty}\sum_{j=0}^{i}a_{i,j}\psi_1^{i-j}p_1^{j},\\
p_2(\psi_1,p_1)&=\sum_{i=1}^{\infty}\sum_{j=0}^{i}b_{i,j}\psi_1^{i-j}p_1^{j}.
\end{aligned} \label{Eqn:PowerSeries}
\end{equation}
To determine the coefficients of the linear terms, note that $\mathcal{W}^C$ is tangent to the plane $E_C=\text{span}\{\mathbf{v}_1,\mathbf{v}_2\}$ where 
\begin{equation}
\mathbf{v}_1=\begin{bmatrix}
    c \\
    1 \\
    0 \\
    0 \\
\end{bmatrix} \text{ and } \mathbf{v}_2=\begin{bmatrix}
    \sigma_1^2/c \\
    0 \\
    1 \\
    0 \\
\end{bmatrix} 
\end{equation}
are, respectively, the eigenvector and generalized eigenvector of the $0$ eigenvalue of $J(\mathcal{O},0)$. Computing the tangent vectors of $\mathcal{W}^C$ in the coordinate directions, we have that
\begin{equation}
\begin{bmatrix}
1\\
a_{1,0}\\
0 \\
b_{1,0}
\end{bmatrix}, 
\begin{bmatrix}
0\\
a_{1,1}\\
1 \\
b_{1,1}
\end{bmatrix}\in \text{span}\left\{\begin{bmatrix}
    c \\
    1 \\
    0 \\
    0 \\
\end{bmatrix},\begin{bmatrix}
    \sigma_1^2/c \\
    0 \\
    1 \\
    0 \\
\end{bmatrix}\right\}
\end{equation}
and thus $a_{1,0}=1/c, a_{1,1}=-\sigma_1^2/c^2$, $b_{1,0}=0$, and $b_{1,1}=0$. 

Since the linear terms in the expansion of $p_2$ were $0$, we need to compute higher order terms to obtain a non-trivial expansion. By the chain rule we have that
\begin{equation}
\begin{aligned}
    \frac{d}{d\tau} \psi_2&=\frac{\partial \psi_2}{\partial \psi_1}\frac{d \psi_1}{d \tau}+\frac{\partial \psi_2}{\partial p_1}\frac{d p_1}{d\tau},\\
    \frac{d}{d\tau} p_2&=\frac{\partial p_2}{\partial \psi_1}\frac{d \psi_1}{d \tau}+\frac{\partial p_2}{\partial p_1}\frac{d p_1}{d\tau},
\end{aligned}
\end{equation}
and thus by Equation \eqref{Eqn:Hamiltonian} we obtain the following system of equations
 \begin{equation}
 \begin{aligned}
g+\sigma_2^2p_2&=\frac{\partial \psi_2}{\partial \psi_1}(f+\sigma_1^2p_1)-\frac{\partial \psi_2}{\partial p_1}\left(\frac{\partial f}{\partial \psi_1}p_1+\frac{\partial g}{\partial \psi_1}p_2\right),\\
-\frac{\partial f}{\partial \psi_2}p_1-\frac{\partial g}{\partial \psi_2}p_2&=\frac{\partial p_2}{\partial \psi_1}(f+\sigma_1^2p_1)-\frac{\partial p_2}{\partial p_1}\left(\frac{\partial f}{\partial \psi_1}p_1+\frac{\partial g}{\partial \psi_1}p_2\right),
\end{aligned}\label{Eqn:InvarianceEqns}
 \end{equation}
where we have suppressed the independent variables to reduce the complexity of the expressions. Therefore, substituting Equation \eqref{Eqn:PowerSeries} into Equation \eqref{Eqn:InvarianceEqns} and equating powers we can obtain linear equations for the undetermined coefficients. Following this procedure, we obtain to cubic order the following approximations
\begin{equation}
\begin{aligned}
\psi_2(\psi_1,p_1)&\approx\frac{1}{c}\psi_1-\frac{(1-\gamma)}{c^5}\psi_1^3-\frac{\sigma_1^2}{c^2}p_1-\frac{3\sigma_1^4}{c^4}p_1^2+\frac{3\sigma_1^2}{c^3}p_1\psi_1,\\
p_2(\psi_1,p_1)&\approx\frac{3(1-\gamma)\sigma_1^4}{c^5}p_1^3+\frac{3(1-\gamma)}{c^3}p_1\psi_1^2.
\end{aligned}\label{Eqn:CenterManifold}
\end{equation}
Note, as expected, on the sub-manifold $\mathbf{p}=0$ we recover the approximation to the center manifold for the deterministic dynamics presented in Section 2. 

To compute a local approximation of the most probable transition path, we now calculate the intersection of the manifold defined by $H=0$ with $\mathcal{W}^C$. To do so, we substitute Equation \eqref{Eqn:CenterManifold} into Equation \eqref{Eqn:HamiltonianScalar} and expand:
\begin{equation}
\begin{aligned}
H(\psi_1,\psi_2(\psi_1,p_1),p_1,p_2(\psi_1,p_1)&=\frac{\sigma_1^2}{2} p_1^2+\frac{\sigma_2^2}{2}p_2^2(\psi_1,p_1)+\langle F(\psi_1,\psi_2(\psi_1,p_1)),(p_1,p_2(\psi_1,p_2)\rangle\\
&\approx\frac{\sigma_1^2}{2}p_1^2-\psi_1^2 p_1.
\end{aligned}
\end{equation}
Consequently, to lowest order, the intersection of the manifold $H=0$ with $\mathcal{W}^C$ corresponds to when $p_1=0$ or $p_1=2\sigma_1^{-2} \psi_1^2$. Therefore, the intersection forms two curves given (locally) by the following parameterizations:
\begin{equation}
\begin{aligned}
(\Psi_1^*(s),\mathbf{p}_1^*(s))&=\left(s,\frac{1}{c}s-\frac{(1-\gamma)}{c^5}s^3,0,0\right),\\
(\Psi_2^*(s),\mathbf{p}_2^*(s))&=\left(s,\frac{1}{c}s-\frac{2}{c^2}s^2-\frac{(1-\gamma)}{c^5}s^3+\frac{6}{c^3}s^3,\frac{2}{c}s^2,\frac{6(1-\gamma)}{c^3 \sigma_1^2}s^4\right).
\end{aligned}
\end{equation}
The curve $(\Psi_1^*(s),\mathbf{p}_1^*(s))$ is simply the local approximation of the center manifold for the deterministic dynamics we found in Section 2. The second curve $(\Psi_2^*(s),\mathbf{p}_2^*(s))$ is the trajectory exiting $(\mathcal{O},0)$ that we are looking for as it is a local approximation of the heteroclinic orbit. Note, in the first two components $(\Psi_1^*,p_1^*)$ and $(\Psi_2^*,p_2^*)$ agree at the linear order and thus, as we suspected, the most probable transition path locally agrees with the center manifold of the deterministic dynamics near the origin. Indeed, in Figure \ref{FIG: overlaid paths} we see that the numerical approximation generated by the gradient flow and $(\Psi_2^*(s),0)$, the projection of the most probable path onto the $\mathbf{p}=0$ plane, are in excellent agreement near $\mathcal{O}$. Furthermore, if we identify the projection $(\Psi_2^*(s),0)$ with its physical coordinates, we obtain the following local approximation to the most probable path
\begin{equation}
m(v)=\frac{1}{c}v-\frac{2}{c^2}v^2-\frac{(1-\gamma)}{c^5}v^3+\frac{6}{c^3}v^3, \label{Eqn:HeteroM_v}
\end{equation}
which, at this order, does not depend on the noise components $\sigma_1,\sigma_2$. 
\begin{figure}[t!]
\centering
\begin{subfigure}[b]{0.45\textwidth}
    \centering
    \includegraphics[width=.9\textwidth]{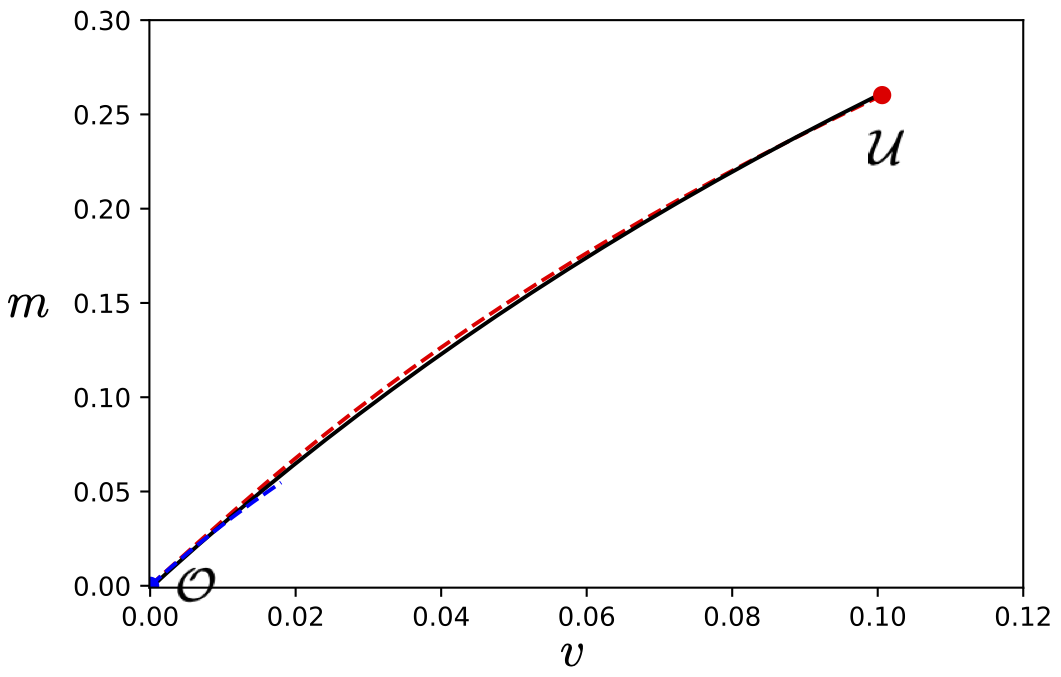}
\end{subfigure}
\begin{subfigure}[b]{0.45\textwidth}
    \centering
\includegraphics[width=.9\textwidth]{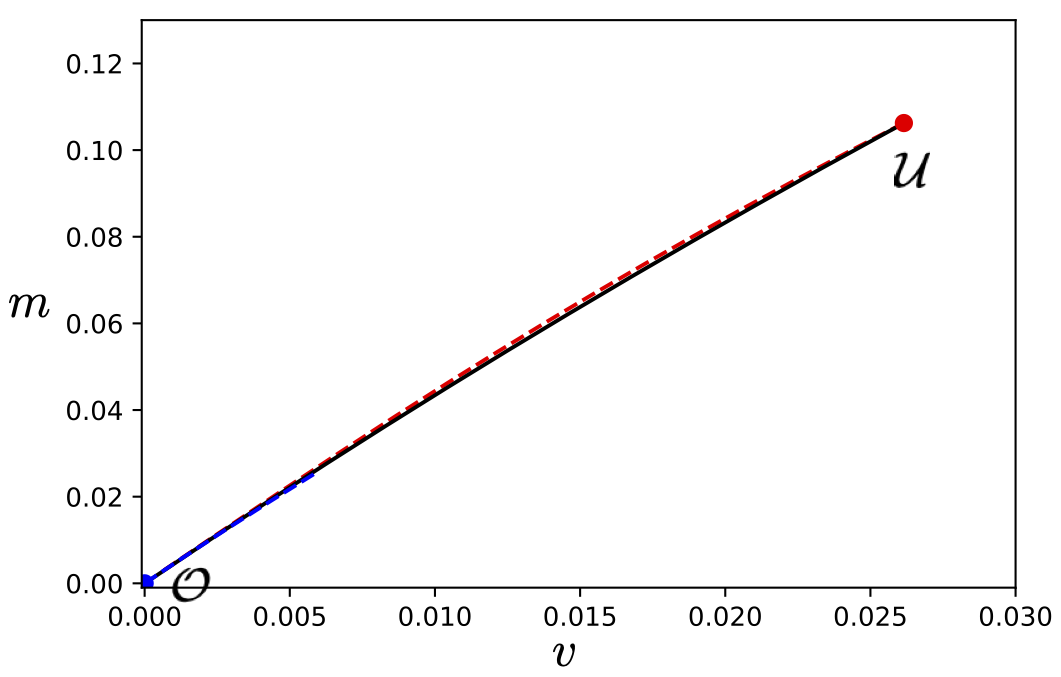}
\end{subfigure}
    \caption{Plots of the most probable path generated from the gradient flow (solid black) overlaid with the unstable manifold of $\mathcal{U}$ (dashed red) and the approximation for the most probable path near $\mathcal{O}$ (dashed blue), for two values of $c$. (a) $c=0.286$. (b) $c=0.22$.}
    \label{FIG: overlaid paths}
\end{figure}

Finally, to obtain an estimate for the expected tipping time we will use the above local approximation to $(\Psi_2^*(s),\mathbf{p}_2^*(s))$, given in physical coordinates by Equation \eqref{Eqn:HeteroM_v}, and Equation \eqref{Eqn:LDP:Expect_Value} to estimate the expected time of leaving a neighborhood near the origin. Specifically, it follows from the asymptotic estimate given by Equation \eqref{Eqn:DeterministicAsyExpansionFP}, that the $v$-coordinate of $\mathcal{U}$ scales like $c^3$ and thus for $0<r<1$ we let $d=rc^3$ serve as a proxy for a typical neighborhood length scale when considering the validity of the approximation given in Equation \eqref{Eqn:HeteroM_v}. If we let $(\Psi^*(s),\mathbf{p}^*(s))$ denote the most probable path satisfying Equation \eqref{Eqn:Hamiltonian}, $\lim_{s\rightarrow -\infty}(\Psi^*(s),\mathbf{p}^*(s))=(\mathcal{O},0)$, and its first component $\psi_1^*$ satisfies $\psi_1^*(0)=d$, then it follows from Equation \eqref{Eqn:Noise:FunctionalHamiltonian} that
\begin{equation}
I[\Psi^*,\mathbf{p}^*]=\int_{-\infty}^0 \left( \left(\frac{p_1^*(s)}{\sigma_1}\right)^2+\left(\frac{p_2^*(s)}{\sigma_2}\right)^2\right)ds.
\end{equation}
Now, if we assume the image of $(\Psi^*(s),p^*(s))$ is locally the same as $(\Psi_2^*(s),\mathbf{p}_2^*(s))$, it follows from Equation \eqref{Eqn:HeteroM_v} that upon changing variables
\begin{equation}
I[\Psi^*,\mathbf{p}^*]=\int_0^{rc^3}\left(\frac{4}{c^2\sigma_1^2}v^4+\frac{36(1-\gamma)^2}{c^6\sigma_1^4\sigma_2^2}v^8\right)\left|\frac{ds}
{dv}\right|dv,
\end{equation}
where by Equation \eqref{Eqn:Hamiltonian} we have that
\begin{equation}
\frac{ds}{dv}=\frac{1}{f(v,m(v))+\frac{2\sigma_1^2}{c}v^2}.
\end{equation}
Therefore, expanding to lowest order, we have that
\begin{equation}
I[\Psi^*,\mathbf{p}^*]=\int_0^{rc^3}\left(\frac{4}{c\sigma_1^2(c-2\sigma_1^2)}v^2+\frac{36(1-\gamma)^2}{c^5\sigma_1^4\sigma_2^2(c-2\sigma_1^2)}v^6\right)dv
\end{equation}
and thus upon integrating we obtain the following scaling law
\begin{equation}
I[\Psi^*,\mathbf{p}^*]= C_1 \frac{c^7}{\sigma_1^2} +C_2\frac{c}{\sigma_2^2}\left(\frac{c^7}{\sigma_1^2}\right)^2,
\end{equation}
where $C_1$, $C_2$ are constants. That is, since $m(v)$ is an approximation for the most probable path, we obtain the following upper bound for the expected escape time from a neighborhood around the origin $\mathcal{O}$:
\begin{equation}
\mathbb{E}[\tau^*]\leq C \exp\left(C_1\frac{c^7}{\sigma_1^2} +C_2\frac{c}{\sigma_2^2}\left(\frac{c^7}{\sigma_1^2}\right)^2\right),\label{Eqn:ExpectedTippingtime}
\end{equation}
where $C>0$ is a sub-exponential correction since Equation \eqref{Eqn:LDP:Expect_Value} is only a logarithmic equivalence.

The scaling law given in Equation \eqref{Eqn:ExpectedTippingtime} is interesting in that it illustrates the interplay between the dimensionless sheer $c$ and $\sigma_1,\sigma_2$. In particular, it provides further evidence for why tipping near the origin is far more common than tipping away from the stable storm state $\mathcal{S}$. Indeed, for the numerical experiments presented in Figure \ref{Fig:Noise:reflecting}, the ratio $c^7/\sigma_1^2$ ranges from (approximately) $1$ to $6$, i.e., is $\mathcal{O}(1)$ despite $\sigma_1=\sigma_2\sim 10^{-3}$.

\section{Discussion} \label{Summary}
An analysis of the various tipping mechanisms for a low-dimensional model of a tropical cyclone show a range of different possibilities for both the formation and destruction of a hurricane. The key results are the following:
\begin{enumerate}
\item The non-storm state $\mathcal{O}$ is a base state that is asymptotically stable for all parameter values. Consequently, there is no possibility of bifurcation-induced tipping away from it to an activated storm state. Furthermore, it does not move in phase space and thus there is no rate-induced tipping that would form a stable storm state from a non-storm state. 
\item The dimensionless wind shear acts as a natural bifurcation parameter: with increasing values of $c$ there is a saddle-node bifurcation in which the stable storm state is eliminated. That is to say that excessive wind shear kills a storm in the sense that it will no longer be able to maintain itself for a prolonged period of time. Thus, in order for the formation of a tropical cyclone to occur, we need sufficiently low wind shear, corroborating physical observations \cite{FR01}.
\item  A necessary condition for the destabilization of the stable storm state, through rate-induced tipping, is that both the potential velocity $V_p$ and dimensionless wind shear $c$ need to be increasing in time. This is counter-intuitive as the energy source is encoded in the maximum potential velocity in this model and thus it would be expected that increasing this quantity might strengthen the cyclone. But we show that, as long as it is accompanied by a sufficient wind shear, it serves to kill the hurricane, and increase of wind shear cannot achieve that alone.  
\item We showed that the non-storm is state is highly susceptible to noise-induced tipping while the stable storm state is robust to random fluctuations. This susceptibility was quantified by the ratio $c^7/\sigma_1^2$ which is a dimensionless measure of the interplay between wind shear and noise. 
\item We identified the most probable transition path for noise-induced tipping from $\mathcal{O}$ closely tracks the center manifold for the deterministic dynamics. That is, the center manifold is a region in phase space that is most vulnerable to random fluctuations. 
\end{enumerate}

As is standard in this type of analysis, we considered these various tipping mechanisms independent of one another. A natural question is how do these various mechanisms couple to induce a storm-state or destabilize a storm. A natural extension of the analysis presented in this paper is to ask how the interplay of a parameter shift and additive noise will affect tipping within the system. Based on the work and analysis in \cite{ritchie2016early} and \cite{slyman2023rate}, we expect that in tipping away from the stable storm state, there will be an interplay between the rate and noise-induced tipping mechanisms, and the additive noise will lower the critical rate needed for tipping. Additionally, tipping should occur away from the non-storm state, but as there was no rate-induced tipping with this initialization, we must explore if there is an interplay of the tipping mechanisms.

Using the same additive noise and ramp parameter as in prior sections, the natural system to study this interaction is given by
\begin{equation}
\begin{aligned}
    dv &= \displaystyle \left(\frac{(1-\gamma) V_p(\Lambda(s))^2}{{V_p^-}^2} m^3-(1-\gamma m^3)v^2\right) d \tau+ \sigma_1 dW_1, \\
    dm &=\left((1-m)v-c(\Lambda(s)) \right)d \tau + \sigma_2 dW_2, \\
    ds &= d\tau.
\end{aligned}
\label{EQ: rate and noise}
\end{equation}
While we would like to use the methods employed in \cite{slyman2023rate} and \cite{fleurantin2023dynamical}, the center manifold of the non-storm state, as described in Section \ref{DetModel}, demands a more careful approach, and is beyond the scope of the current study. Nevertheless, when conducting Monte-Carlo simulations of \eqref{EQ: rate and noise} we noticed some interesting phenomena that have led to further conjectures. First, initializing at the stable storm state corresponding to the start of the ramp function, we see in Figure \ref{fig: rate and noise}(a) that for $r<r_c$, with the addition of noise, there is tipping to the non-storm state. Notice there is an interplay between these two mechanisms. Initializing at the non-storm state corresponding to the start of the ramp, we see in Figure \ref{fig: rate and noise}(b) that there is tipping to the stable storm state. Note, in Figure \ref{fig: rate and noise}(b), realizations that tip actually begin by tipping to the stable storm state corresponding to the start of the ramp function, and then end-point track to the stable storm state corresponding to the end of the ramp function. This implies the two tipping mechanisms are not interacting to induce tipping: noise-induced tipping occurs first, followed by tracking of the stable storm state from the parameter shift. While we illustrate this phenomenon for one set of noise strengths and a rate parameter, this behavior held true for multiple sets of parameter values.

\begin{figure}
\centering
\begin{subfigure}[b]{0.4\textwidth}
    \centering
    \includegraphics[width=.95\textwidth]{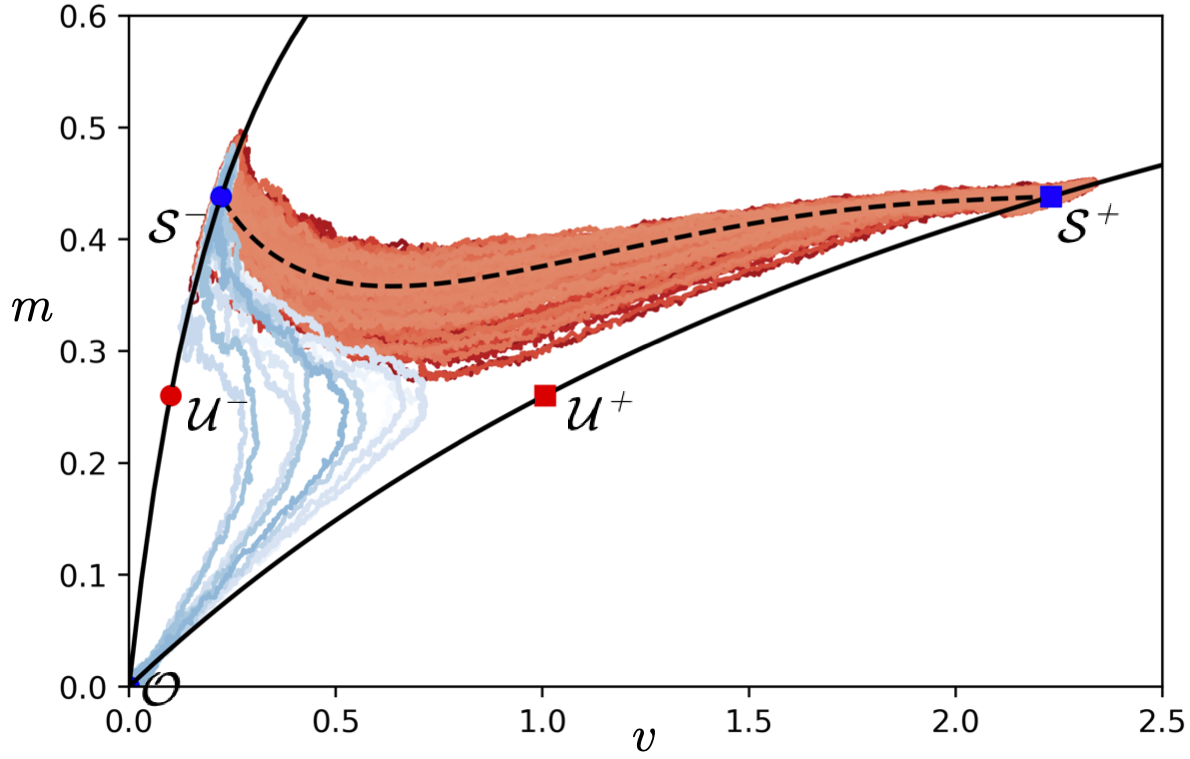}
    \caption{}
\end{subfigure}
\begin{subfigure}[b]{0.4\textwidth}
    \centering
    \includegraphics[width=.95\textwidth]{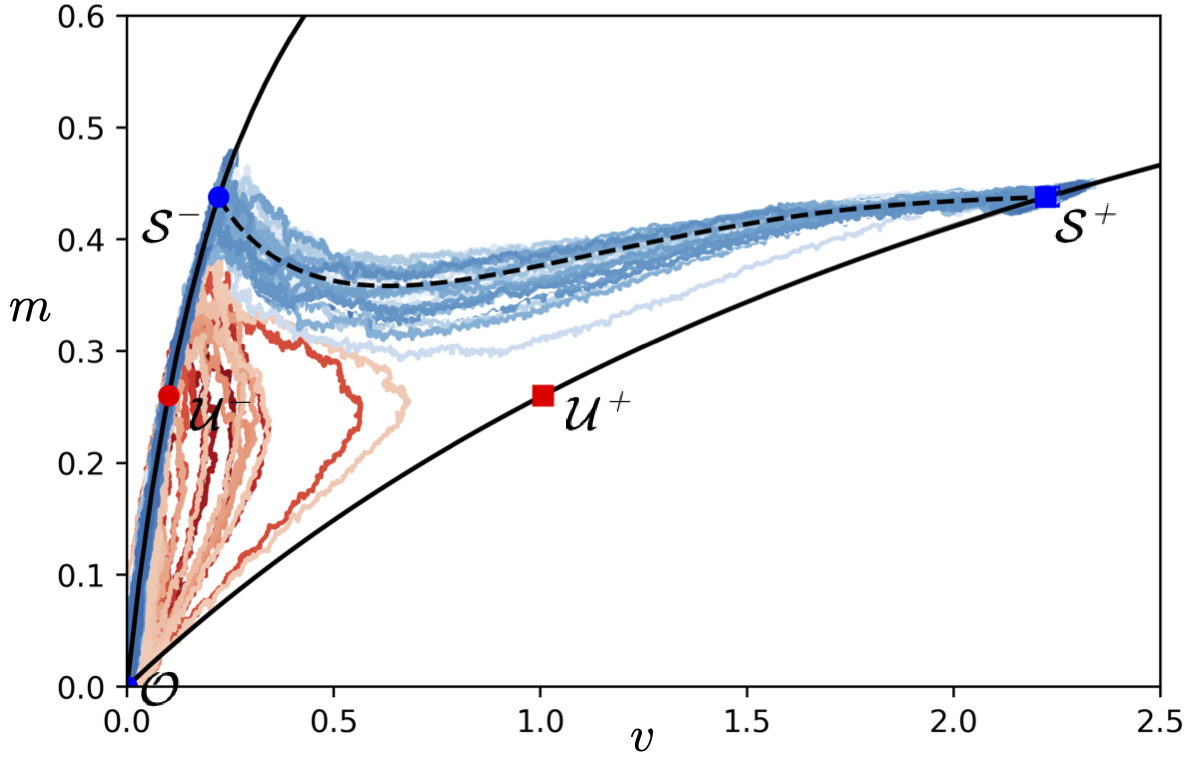}
    \caption{}
\end{subfigure}
    \caption{Realizations of the system given in Equation \eqref{EQ: rate and noise}, solved with the Euler-Maruyama method with $\tau_f=2500, d \tau=0.1$. Parameter values are set as $\sigma_1=\sigma_2=0.005,r=.03,\gamma=0.43$, and $V_p,c$ are time-dependent and defined as in Equations \eqref{ramp_ex}, \eqref{VP_ex}, and \eqref{S_ex}. The three fixed points at the start of the ramp, $\mathcal{O}, U^-, S^-$, correspond to the non-storm state, the unstable storm state, and the stable storm state. These stable fixed points are denoted by blue circles and the saddle node is denoted by a red circle. The three fixed points at the end of the ramp, $\mathcal{O}, U^+, S^+$, correspond to the non-storm state, the unstable storm state, and the stable storm state. These stable fixed points are denoted by blue squares and the saddle node is denoted by a red square. The solid black curves correspond to the $\dot{m}$ nullcline at the start and the end of the ramp function. The dashed black curves correspond to the solution in the deterministic system ($\sigma_1=\sigma_2=0$). 
    (a) 1000 realizations initialized at the stable storm state, $S^-$.  The blue realizations tip from $S^-$ to $\mathcal{O}$. The red realizations end-point track the stable path from $S^-$ to $S^+$, and do not tip. 
    (b) 1000 realizations initialized at the non-storm state, $\mathcal{O}$. The blue realizations tip from $\mathcal{O}$ to $S^+$. The red realizations do not tip.}
    \label{fig: rate and noise}
\end{figure}

To validate the above conjectures requires an adaptation of the standard Freidlin-Wentzell theory of large deviations to non-autonomous systems. One approach is to first follow the procedure in \cite{wieczorek2021compactification} and ``compactify'' the system in such a way that compact invariant sets such as equilibria now describe the long time behavior of the system. Noise-induced tipping can now be studied on the compactified system using the standard Freidlin-Wentzell theory of large deviations. When considering the various tipping phenomenon, this system will now contain multiple scales, e.g. $r$, $c^7/\sigma_1^2$, and $c/c_2^2$, and it would be interesting to understand how these various parameters influence tipping phenomenon. From a variational perspective, $\Gamma$-convergence provides a natural tool for studying the minimizers of the Freidlin-Wentzell functional in various asymptotic limits.

\section*{Acknowledgement}
The authors would like to acknowledge support of the Mathematics and Climate Research Network (MCRN) under NSF grant DMS-1722578. Katherine Slyman and Christopher Jones were additionally supported by Office of Naval Research grant N000141812204.

\appendix
\label{AppendixA}
\section{Center Manifold Approximation of the Origin}
To determine an approximation of the center manifold at $\mathcal{O}$ for Equation \eqref{Eqn:DimensionlessAut} we follow \cite{wiggins2003introduction} and consider a solution of the form
\begin{equation}
m=h(v)=\Sigma_{k=1}^n a_k v^k.
\label{Eqn:center_manifold_form}
\end{equation}
Differentiating Equation \eqref{Eqn:center_manifold_form} with respect to $v$, it follows from Equation \eqref{Eqn:DimensionlessAut} that
\begin{equation}
g(v,h(v))=\frac{dm}{d\tau}=h_v(v)\frac{dv}{d\tau}=h_v(v) f(v,h(v)).
\end{equation}
However from \eqref{Eqn:DimensionlessAut} we also know that $\dot{m}=g(v,h(v))$. Therefore, assuming $n=5$ and equating powers of $v$ in the equation $h_v(v) f(v,h(v))=g(v,h(v))$ we arrive at
\begin{equation}
\begin{aligned}
a_1&=\frac{1}{c}, \\
a_2&=0,\\
a_3&=\frac{\gamma-1}{c^5},\\
a_4&=\frac{2(\gamma-1)}{c^6},\\
a_5&=\frac{6-6c^2-12\gamma+6c^2\gamma-c^4\gamma+6\gamma^2}{c^9}
\end{aligned}
\end{equation}
and thus $m(v)=a_1v+a_2v^2+a_3v^3+a_4v^4+a_5v^5$ is the desired approximation of the center manifold for \eqref{Eqn:DimensionlessAut} near $\mathcal{O}$.  

\section{Stability of the Origin}
From Appendix A, we found a center manifold approximation near $\mathcal{O}$ given by 
\begin{equation}
m(v)=\frac{1}{c} v+\frac{\gamma-1}{c^5} v^3+\frac{2(\gamma-1)}{c^6}v^4 + O(v^5).
\end{equation}
If we consider Equation \eqref{Eqn:DimensionlessAut}, we can use $\frac{dv}{d\tau}$ as a differential equation for the dynamics of the center manifold by replacing $m$ with $m(v)$, resulting in
\begin{equation}
\begin{aligned}
    \frac{dv}{d\tau}&=(1-\gamma)(m(v))^3-(1-\gamma(m(v))^3)v^2\\
    &=-v^2 + O(v^3).
\end{aligned}
\end{equation}
Since the coefficient of $v^2$ is negative, it follows that the origin is an asymptotically stable fixed point for the center manifold and hence for the original system in Equation \eqref{Eqn:DimensionlessAut}.

\section{Proof of Proposition \ref{Prop_FI}}

\begin{proof} We want to show that the box $K_{a,b}=[a_1,b_1]\times[a_2,b_2]$ is forward invariant with respect to the flow, and therefore we need to find the direction of the flow on the four sides of the box.
\begin{itemize} 
\item Side 1 ($v=a_1$, $m\in[a_2,b_2]$):
\begin{equation}
\begin{aligned}
\dot{v}&=(1-\gamma)\left(\frac{V_p}{V_p^-}\right)^2m^3-(1-\gamma m^3)a_1^2 \\
&=m^3\left((1-\gamma)\left(\frac{V_p}{V_p^-}\right)^2+\gamma a_1^2\right)-a_1^2 \\
&\geq a_2^3\left((1-\gamma)\left(\frac{V_p}{V_p^-}\right)^2+\gamma a_1^2\right)-a_1^2 \\
&>0, 
\end{aligned}
\end{equation}
since $\sqrt[3]{\frac{a_1^2}{{(1-\gamma)}\left(\frac{V_p}{V_p^-}\right)^2+\gamma a_1^2}}<a_2$.
\item Side 2 ($v\in[a_1,b_1]$, $m=a_2)$:
\begin{equation}
\begin{aligned}
\dot{m}&=(1-a_2)v-ca_2\\
&\geq (1-a_2)a_1-ca_2 \\
&=a_1-(a_1+c)a_2\\
&>0,
\end{aligned}
\end{equation}
since $a_2<\frac{a_1}{a_1+c}$.

\item Side 3 ($v=b_1$, $m\in[a_2,b_2]$):
\begin{equation}
\begin{aligned}
\dot{v}&=(1-\gamma)\left(\frac{V_p}{V_p^-}\right)^2m^3-(1-\gamma m^3)b_1^2 \\
&=m^3\left((1-\gamma)\left(\frac{V_p}{V_p^-}\right)^2+\gamma b_1^2\right)-b_1^2 \\
&\leq b_2^3\left((1-\gamma)\left(\frac{V_p}{V_p^-}\right)^2+\gamma b_1^2\right)-b_1^2 \\
&<0, 
\end{aligned}
\end{equation}
since  $b_2<\sqrt[3]{\frac{b_1^2}{{(1-\gamma)}\left(\frac{V_p}{V_p^-}\right)^2+\gamma b_1^2}}$.
\item Side 4 ($v\in[a_1,b_1]$, $m=b_2$):
\begin{equation}
\begin{aligned}
\dot{m}&=(1-b_2)v-cb_2\\
&\leq (1-b_2)b_1-cb_2 \\
&=b_1-(b_1+c)b_2\\
&<0,
\end{aligned}
\end{equation}
since $\frac{b_1}{b_1+c}<b_2$.
 \end{itemize}
\end{proof}

\bibliographystyle{elsarticle-num}


\end{document}